\documentclass[11pt,dvipsnames]{amsart}

\usepackage{xy}
\usepackage{graphicx, amsmath, amssymb, amsthm, amsfonts}
\usepackage{hyperref}
\hypersetup{colorlinks, citecolor=Red, linkcolor=NavyBlue, urlcolor=NavyBlue}
\usepackage[capitalise, nameinlink]{cleveref}
\usepackage[margin=1.2in]{geometry}
\usepackage{mathtools}
\usepackage{quiver}
\usepackage[textsize=tiny]{todonotes}
\usepackage{microtype}

\usepackage[backend=biber,style=alphabetic]{biblatex}

\newcommand{\C}{\mathbb{C}}
\newcommand{\Z}{\mathbb{Z}}
\newcommand{\N}{\mathbb{N}}
\newcommand{\A}{\mathcal{A}}

\DeclareMathOperator{\colim}{colim}
\DeclareMathOperator{\id}{id}
\DeclareMathOperator{\Sym}{Sym}
\DeclareMathOperator{\Fun}{Fun}
\DeclareMathOperator{\res}{res}
\DeclareMathOperator{\tr}{tr}
\DeclareMathOperator{\nm}{nm}
\DeclareMathOperator{\conj}{c}
\DeclareMathOperator{\cod}{cod}
\DeclareMathOperator{\ev}{ev}
\DeclareMathOperator{\coeq}{coeq}

\DeclarePairedDelimiter{\of}{(}{)}

\newcommand{\Tamb}{\mathsf{Tamb}}
\newcommand{\Mack}{\mathsf{Mack}}
\newcommand{\Ring}{\mathsf{Ring}}
\newcommand{\CRing}{\mathsf{CRing}}
\newcommand{\Green}{\mathsf{Green}}
\newcommand{\Alg}{\mathsf{Alg}}
\newcommand{\Mod}{\mathsf{Mod}}
\newcommand{\set}{\mathsf{set}}
\newcommand{\Set}{\mathsf{Set}}
\newcommand{\catC}{\mathsf{C}}
\newcommand{\catD}{\mathsf{D}}
\newcommand{\Poly}{\mathcal{P}}
\newcommand{\Cat}{\mathsf{Cat}}
\newcommand{\GCRing}[1][G]{\CRing^{#1}}
\newcommand{\GTamb}[1][G]{\Tamb_{#1}}
\newcommand{\GMack}[1][G]{\Mack_{#1}}
\newcommand{\GGreen}[1][G]{\Green_{#1}}
\newcommand{\Gset}[1][G]{\Set^{#1}}
\newcommand{\kAlg}[1][k]{\Alg_{#1}}
\newcommand{\kMod}[1][k]{\Mod_{#1}}

\newcommand{\toby}[1]{\overset{{#1}}{\rightarrow}}

\newtheorem{theorem}{Theorem}[section]
\newtheorem{letterthm}{Theorem}

\newtheorem{lemma}[theorem]{Lemma}

\newtheorem{definition}[theorem]{Definition}
\newtheorem{proposition}[theorem]{Proposition}
\newtheorem{corollary}[theorem]{Corollary}

\newtheoremstyle{BoldRemark} 
{10pt}                    
{10pt}                    
{\upshape}                   
{}                           
{\bfseries}                  
{.}                          
{.5em}                       
{}  
\theoremstyle{BoldRemark}

\newtheorem{example}[theorem]{Example}

\newtheoremstyle{DefinitionOf}
{\topsep}   
{\topsep}   
{\normalfont}  
{0pt}       
{\bfseries} 
{.}         
{5pt plus 1pt minus 1pt} 
{\thmname{#1}\thmnumber{ \normalfont{}(Part #2)}}          
\theoremstyle{DefinitionOf}
\newtheorem{def_of_poly}{Definition of $\Poly_G$}

\mathchardef\mhyphen="2D

\title{Algebraically Closed Fields in Equivariant Algebra}
\author{Jason Schuchardt \and Ben Spitz \and Noah Wisdom}
\date{}

\begin{document}

\begin{abstract}
    Using the Burklund-Schlank-Yuan abstraction of ``algebraically closed" to ``Nullstellensatzian", we show that a $G$-Tambara functor is Nullstellensatzian if and only if it is the coinduction of an algebraically closed field (for any finite group $G$). As a consequence we deduce an equivalence between the $K$-theory spectrum of any Nullstellensatzian $G$-Tambara functor with the $K$ theory of some algebraically closed field.
\end{abstract}
\maketitle

\tableofcontents

\section{Introduction}

In \cite{BSY22} the notion of a Nullstellensatzian object in a category is introduced in order to abstract the definition of an algebraically closed field. A ring is an algebraically closed field if and only if it is a Nullstellensatzian object in the category of rings, and the central aim of \cite{BSY22} is to classify the Nullstellensatzian objects in a certain category of great importance to homotopy theorists. In particular, there is a functor for which the Nullstellensatzian objects are precisely those obtained by applying the functor to an algebraically closed field of characteristic $p$. The authors of \cite{BSY22} obtain many striking applications of their results. For example, under mild conditions, every object in a category admits a morphism to a Nullstellensatzian object (in some sense this is a generalization of taking an algebraic closure). This immediately allows the authors of \cite{BSY22} to reduce the remaining part of the proof of the famed red-shift conjecture to merely establishing a red-shift property for Nullstellensatzian objects in their setting, which (along with \cite{Yua21}) completes the proof of the red-shift conjecture.

In equivariant homotopy theory, homotopy groups of highly structured ring spectra give Tambara functors rather than rings. One might hope to prove equivariant analogues of the results of \cite{BSY22}; such results would likely be useful in the study of chromatic equivariant homotopy theory, especially in relation to the Balmer spectrum computations of \cite{BGH20} \cite{BHN+19}, or towards the problem of studying motivic height \cite{BGL22}. As a stepping stone towards this, we first study the purely algebraic case: in this article, we completely classify the algebraically closed objects in the category of Tambara functors. Additionally, we classify Nullstellensatzian $G$-rings, and compute algebraic closures.

Roughly speaking, the impetus of equivariant algebra is to port over results from commutative algebra, broadly construed, to Tambara functors. For example, Nakaoka \cite{Nak11a} \cite{Nak11b} has defined ideals, domains, fields, and localizations of Tambara functors, Blumberg and Hill \cite{BH18} define free polynomial algebras over a Tambara functor (which are observed to be almost never flat by Hill-Merhle-Quigley \cite{HMQ22}), and Hill \cite{Hil17} has established a robust theory of Andre-Quillen homology and K\"{a}hler differentials. In \cite{4DS24}, the second author, along with Chan, Mehrle, Quigley, and Van Niel, initiate a study of affine schemes of Tambara functors, and prove an analogue of the going up theorem, as well as a weak version of a Hilbert basis theorem. Also, a complete description of all field-like Tambara functors for $G = \mathbb{Z}/p^n$ is given by the third author in \cite{Wis24}.

Recall from \cite{BSY22} the definition of a Nullstellensatzian object in a category.

\begin{definition}
    A category $\catD$ with an initial object is Nullstellensatzian if every compact, nonterminal object admits a map to the initial object. An object $x$ in a category $\catC$ is Nullstellensatzian if the slice category $x/\catC$ of $x$-algebras is Nullstellensatzian.
\end{definition}

Our main result is a complete classification of Nullstellensatzian objects in the category of $G$-Tambara functors, for any finite group $G$. We remark that complete classifications of Nullstellensatzian objects are sparse in the literature; to the author's knowledge, the only other categories for which Nullstellensatzian objects are completely characterized are the category of rings (by Hilbert's Nullstellensatz), and the category of $T(n)$-local $\mathbb{E}_\infty$-ring spectra \cite{BSY22} (although it seems that the notion of ``algebraically closed group" studied, for example, in \cite{Sco51} \cite{Mac72} \cite{Bel74} is very closely related to the notion of Nullstellensatzian group).

\begin{letterthm}[cf. \Cref{theorem:alg-closed-iff-coinduced}]
    Let $k$ be a Nullstellensatzian Tambara functor. Then $k$ is isomorphic to the coinduction of an algebraically closed field, i.e., $k$ is the fixed-point Tambara functor associated to the $G$-ring $\Fun(G,F)$ for $F$ algebraically closed.
\end{letterthm}

In fact, it turns out that $k$ and $F$ are Morita equivalent: their respective categories of modules are equivalent (this is a very special case of a much broader result, see \cite{Wis25} for a more complete discussion). Thus we deduce the following.

\begin{letterthm}[cf. \Cref{cor:K-theory-equiv}]
    Let $k$ be a Nullstellensatzian Tambara functor. Then we have an equivalence of $\mathbb{E}_\infty$-rings $K(k) \cong K(k(G/G))$ of $K$-theory spectra.
\end{letterthm}

In particular, the computations of Suslin \cite{Sus83} towards the $K$-groups of algebraically closed fields also apply verbatim to the $K$-groups of Nullstellensatzian $G$-Tambara functors. To the author's knowledge, this is the first identification of the $K$-theory spectrum of (the Green functor underlying) a Tambara functor, although we are unable to fully determine its homotopy groups. We also direct the reader to \cite{CW25} for the first full computation of the $K$-theory spectrum, along with all of the $K$-groups, of a Tambara functor for which the previously mentioned Morita equivalence argument does not work.

We note that the study of modules over (the Green functor underlying) a Tambara functor, particularly their homological algebra, has a rich history. For example, \cite{HMQ22} observe that free polynomial algebras over Tambara functors tend to fail to be flat as modules, in contrast to the classical picture, where $R[x]$ is a free $R$-module, for any ring $R$. As a consequence of our Morita equivalence result, free polynomial algebras over Nullstellensatzian Tambara functors are always flat; in fact, they are always free. Additionally, \cite{BSW17} show that the category of so-called cohomological $G$-Mackey functors turns out to have finite global projective dimension with surprising frequency; our Morita equivalence result implies that (the Green functor underlying) any Nullstellensatzian Tambara functor has global projective dimension $0$. Lastly, \cite{Hil17} studies square-zero extensions in the context of modules over Tambara functors, and we expect the story of \cite{Hil17} to simplify enormously over Nullstellensatzian Tambara functors.

We now turn our attention to the proof of our main result. We may control the behavior of Nullstellensatzian objects by controlling compact objects. It turns out that the compact objects in the category of Tambara functors are precisely the finitely presented Tambara functors (this is a completely formal fact from the theory of universal algebra, which we exposit in \Cref{prop:compact same as fp}). Thus, in particular, the free Tambara functor on a finite $G$-set is always compact. However, we will require many more examples of compact Tambara functors, and for this we prove the following theorem.

\begin{letterthm}[cf. \Cref{thm:constant-preserves-compacts}]
    The functor $S \mapsto \underline{S} : \CRing \to \Tamb_G$ preserves compact objects and compact morphisms. In other words, $\underline{\Z}$ is a finitely presented Tambara functor, and if $R \to S$ is a finitely presented morphism of rings, then $\underline{R} \to \underline{S}$ is a finitely presented morphism of Tambara functors.
\end{letterthm}

One of the key ingredients of this result is, in turn, the fact that the free Tambara functor $\A[x_G]$ is Noetherian; this an equivariant incarnation of the Hilbert basis theorem. We establish this by proving the following stronger statement, which generalizes the case $G = C_p$ proven in \cite[Proposition 3.9]{4DS24}.

\begin{letterthm}[cf. \Cref{prop:free-at-fixed-pts-is-levelwise-fg}]
    The free Tambara functor $\A[x_G]$ on a generator at level $G/G$ is levelwise finitely generated as a commutative ring.
\end{letterthm}

With these tools, we can establish necessary conditions for a Tambara functor to be Nullstellensatzian; namely, that any Nullstellensatzian Tambara functor is isomorphic to one of the claimed form. We then prove our main theorem by establishing that everything of this form is, in fact, Nullstellensatzian.

A formal consequence of our classification of Nullstellensatzian Tambara functors is that the Nullstellensatzian $G$-rings are precisely those of the form $\Fun(G,F)$ with $F$ algebraically closed. Additionally, from this we may describe algebraic closures of field-like Tambara functors. From this, one might expect to be able to study Galois theory of Tambara fields, because normal and separable field extensions may both be defined purely in terms of algebraic closures, without any reference to elements. However, as observed in \cite{Wis24}, there are not really enough examples of field-like Tambara functors (at least for $G = \mathbb{Z}/p^n$) to give rise to interesting examples.

Finally, one might have hoped for a more exotic answer to the question ``what are the Nullstellensatzian $G$-Tambara functors". We might therefore have worked in a sufficiently large and well-behaved full subcategory of $\Tamb$ consisting of Tambara functors which are not coinduced. In this situation, if a Nullstellensatzian Tambara functor $R$ is field-like, then each ring $R(G/H)$ would be an honest field, and one would hope for Nullstellensatzian objects in this setting to give rise to interesting finite group actions on fields. 

Unfortunately, this method turns out not to yield a nice answer. Indeed, in \cite{Wis25} the third-named author constructs such a category of $G$-Tambara functors essentially by formally declaring coinductions to be zero. It then turns out that if $R$ is a Nullstellensatzian object in the resulting category, and $R$ is field-like, then $R(G/e)$ is an algebraically closed field, which by Artin-Schreier theory forces it to have characteristic $0$ and forces the $G$-action on $R(G/e)$ to factor through $\C_2$. On the other hand, it is observed in \cite{Wis25} that there are many Nullstellensatzian objects which are not field-like; their study remains open and presumably related to our initial motivation of the study of an equivariant chromatic Nullstellensatz.

\subsection{Contents}

We begin with a review of Tambara functors in section 2. Next, we recall the relevant results of Nakaoka \cite{Nak11a} on field-like Tambara functors and Blumberg-Hill \cite{BH18} on free polynomial Tambara functors in section 3. In section 4, we state from \cite{BSY22} the definition of a Nullstellensatzian object in a category $\mathsf{C}$, and study how compact and Nullstellensatzian objects interact with adjoint functors. Finally, in section 5, we characterize the Nullstellensatzian objects in the category of Tambara functors.

\subsection{Notation and Terminology}

When there is only one group in play, we may abbreviate
``$G$-Tambara functor'' and ``$G$-Mackey functor'' simply as
``Tambara functor'' and ``Mackey functor.'' The category of Tambara
functors will be denoted $\GTamb$ or $\Tamb$, and the category of Mackey
functors will be denoted $\GMack$ or $\Mack$. If $k$ is a Tambara functor,
we will use term \emph{$k$-algebra} to refer to an object of the slice
category $k/\Tamb$, i.e., a morphism $k \to T$ for some Tambara functor
$T$. We will typically use
$\kAlg$ to denote the category of $k$-algebras.

For $R$ a commutative $G$-ring, $\underline{R}$ will denote the fixed-point Tambara functor of $R$.

We will often refer to a Nullstellensatzian object as algebraically closed. Since we are working solely with field-like Tambara functors, there is no risk of confusion with other possible definitions of algebraically closed objects (in particular in terms of nonexistence of finite \'{e}tale extensions).

\subsection{Acknowledgements}

The authors would like to thank Mike Hill for many things, including suggesting a simplified argument for \cref{exl:constant-C-is-not-ac}, as well as Tomer Schlank for posing this question, and Allen Yuan, David Chan, and Haochen Cheng for helpful discussions. The authors also thank J.D. Quigley for insightful comments on an early draft.

Schuchardt and Spitz were partially supported by NSF grant DMS-2136090.
\section{Primer on Tambara Functors}

Tambara functors are an ``equivariant generalization'' of commutative rings---for each finite group $G$, there is a notion of ``$G$-Tambara functor,'' and in the case that $G$ is the trivial group, this notion coincides exactly with that of a commutative ring. For the unfamiliar reader, a thorough treatment of the theory of Tambara functors is given in \cite{Str12}. We give here a streamlined introduction, containing only the necessary details to state the relevant definitions.

We use $\Gset$ to denote the category of finite $G$-sets, and $\Set$ to denote the category of all sets. Eventually, we will define a $G$-Tambara functor to be a certain kind of functor $\Poly_G \to \Set$ from the category $\Poly_G$ to be defined below. The first step towards constructing $\Poly_G$ is the following observation about slices of $\Gset$:

\begin{proposition}
    Let $f : X \to Y$ be a morphism of finite $G$-sets. The functor $\Sigma_f : \Gset/X \to \Gset/Y$ defined by $g \mapsto f \circ g$ sits in an adjoint triple
    \[\begin{tikzcd}[ampersand replacement=\&]
            {\Gset/X} \\
            \\
            {\Gset/Y}
            \arrow["{\Sigma_f}"', shift right=5, from=1-1, to=3-1]
            \arrow["{\Pi_f}", shift left=5, from=1-1, to=3-1]
            \arrow["{f^*}"{description}, from=3-1, to=1-1]
        \end{tikzcd}\]
    where $f^*$ is given by pullback along $f$, and $\Pi_f$ is the functor defined on objects by
    \[\Pi_f(g : A \to X) := \left\{(y, t) : y \in Y, t \in 
    \Gamma(g,f^{-1}(y))
    \right\} \xrightarrow{(y,t) \mapsto y} Y,\]
    where for $g: Z\to X$, and $A\subseteq X$, 
    \[
        \Gamma(g,A) := \{t \in \Set(A, Z) \mid g\circ t = 1_A \} 
        = \prod_{a\in A} g^{-1}(a)
    \]
    is the set of sections of $g$ on the subset $A$.
    The $G$-action is given by
    $g(y,t) = (gy, x\mapsto gt(g^{-1}x))$.
\end{proposition}

Now we can begin defining $\Poly_G$.

\begin{def_of_poly}\label{def:poly_g part 1}
    In the category $\Poly_G$,
    \begin{itemize}
        \item objects are finite $G$-sets;
        \item morphisms $X \to Y$ are isomorphism classes of \emph{bispans}, i.e. diagrams
              \[X \leftarrow \bullet \to \bullet \to Y\]
              in $\Gset$, where two diagrams $X \leftarrow A \to B \to Y$ and $X \leftarrow A' \to B' \to Y$ are said to be isomorphic if there exist isomorphisms $A \to A'$ and $B \to B'$ in $\Gset$ making
              \[\begin{tikzcd}[ampersand replacement=\&,row sep=tiny]
                      \& A \& B \\
                      X \&\&\& Y \\
                      \& {A'} \& {B'}
                      \arrow[from=1-2, to=1-3]
                      \arrow[from=1-2, to=2-1]
                      \arrow["\cong"{description}, from=1-2, to=3-2]
                      \arrow[from=1-3, to=2-4]
                      \arrow["\cong"{description}, from=1-3, to=3-3]
                      \arrow[from=3-2, to=2-1]
                      \arrow[from=3-2, to=3-3]
                      \arrow[from=3-3, to=2-4]
                  \end{tikzcd}\]
              commute. We use $[X \leftarrow A \to B \to Y]$ to denote the isomorphism class of the diagram $X \leftarrow A \to B \to Y$.
    \end{itemize}
\end{def_of_poly}

We must now describe the composition operation in $\Poly_G$. First, we name three classes of distinguished morphisms.

\begin{definition}
    Let $f : X \to Y$ be a morphism in $\Gset$. We declare
    \begin{align*}
        T_f & = [X \xleftarrow{\id} X \xrightarrow{\id} X \xrightarrow{f} Y] \\
        N_f & = [X \xleftarrow{\id} X \xrightarrow{f} Y \xrightarrow{\id} Y] \\
        R_f & = [Y \xleftarrow{f} X \xrightarrow{\id} X \xrightarrow{\id} X]
    \end{align*}
\end{definition}

Tambara functors are also known as TNR functors, referring to these distinguished classes of morphisms.

\begin{def_of_poly}
    For any diagram $X \xleftarrow{h} A \xrightarrow{g} B \xrightarrow{f} Y$ in $\Gset$, we declare
    \[T_f \circ N_g \circ R_h = [X \xleftarrow{h} A \xrightarrow{g} B \xrightarrow{f} Y]\]
    in $\Poly_G$.
\end{def_of_poly}

We will next define pairwise compositions for morphisms in the three distinguished classes $T$, $N$, and $R$. Within individual classes, the composition is straightforward.

\begin{def_of_poly}
    For any diagram $X \xrightarrow{f} Y \xrightarrow{g} Z$ in $\Gset$, we declare
    \begin{align*}
        T_g \circ T_f & = T_{g \circ f} \\
        N_g \circ N_f & = N_{g \circ f} \\
        R_f \circ R_g & = R_{g \circ f}
    \end{align*}
    in $\Poly_G$.
\end{def_of_poly}

When the morphisms are in $T$ and $R$, or in $N$ and $R$, 
composition is also easy to define.

\begin{def_of_poly}
    For any pullback square
    \[\begin{tikzcd}[ampersand replacement=\&]
            W \& Y \\
            X \& Z
            \arrow["{f'}", from=1-1, to=1-2]
            \arrow["{g'}"', from=1-1, to=2-1]
            \arrow["\lrcorner"{anchor=center, pos=0.125}, draw=none, from=1-1, to=2-2]
            \arrow["g", from=1-2, to=2-2]
            \arrow["f"', from=2-1, to=2-2]
        \end{tikzcd}\]
    in $\Gset$, we declare
    \begin{align*}
        R_g \circ T_f & = T_{f'} \circ R_{g'} \\
        R_g \circ N_f & = N_{f'} \circ R_{g'}
    \end{align*}
    in $\Poly_G$.
\end{def_of_poly}

The most complicated composition rule is between morphisms in the 
$T$ and $N$ classes.

\begin{def_of_poly}
    Consider a diagram $X \xrightarrow{f} Y \xrightarrow{g} Z$ in $\Gset$, and view $f$ as an object of $\Gset/Y$. Apply $\Pi_g$ to $f$ and apply $g^*$ to $\Pi_g f$ to obtain a diagram
    \[\begin{tikzcd}[ampersand replacement=\&]
            \& \bullet \& \bullet \\
            X \& Y \& Z
            \arrow["{g'}", from=1-2, to=1-3]
            \arrow["{g^* \Pi_g f}", from=1-2, to=2-2]
            \arrow["\lrcorner"{anchor=center, pos=0.125}, draw=none, from=1-2, to=2-3]
            \arrow["{\Pi_g f}", from=1-3, to=2-3]
            \arrow["f"', from=2-1, to=2-2]
            \arrow["g"', from=2-2, to=2-3]
        \end{tikzcd}\]
    in $\Gset$. Now the counit $\varepsilon$ of the adjunction $g^* \dashv \Pi_g$ yields a morphism $g^* \Pi_g f \to f$ in $\Gset/Y$, giving us a diagram
    \[\begin{tikzcd}[ampersand replacement=\&]
            \& \bullet \& \bullet \\
            X \& Y \& Z
            \arrow["{g'}", from=1-2, to=1-3]
            \arrow["{\varepsilon_f}"', from=1-2, to=2-1]
            \arrow["{g^* \Pi_g f}", from=1-2, to=2-2]
            \arrow["\lrcorner"{anchor=center, pos=0.125}, draw=none, from=1-2, to=2-3]
            \arrow["{\Pi_g f}", from=1-3, to=2-3]
            \arrow["f"', from=2-1, to=2-2]
            \arrow["g"', from=2-2, to=2-3]
        \end{tikzcd}\]
    in $\Gset$. We declare
    \[N_g \circ T_f = T_{\Pi_g f} \circ N_{g'} \circ R_{\varepsilon_f}\]\
    in $\Poly_G$.
\end{def_of_poly}

One must check that the above definitions of pairwise compositions are coherent, i.e., respect the isomorphism relation of Part~\labelcref{def:poly_g part 1} of the definition, and do not force additional relations between diagrams. This is indeed true (see \cite[Lemma 2.15]{gambino-kock}), and with this in place, the composition operation in $\Poly_G$ is fully defined, since an arbitrary composition of two morphisms can be reduced via:
\begin{align*}
    (TNR)(TNR) & \rightsquigarrow TN(TR)NR  \\
               & \rightsquigarrow T(TNR)RNR \\
               & \rightsquigarrow TNRNR     \\
               & \rightsquigarrow TN(NR)R   \\
               & \rightsquigarrow TNR.
\end{align*}

The identity arrow of an object $X \in \Poly_G$ is $R_{\id_X} = N_{\id_X} = T_{\id_X}$.

\begin{proposition}
    $\Poly_G$ admits finite products, given on objects by disjoint union of $G$-sets.
\end{proposition}

\begin{definition}
    A $G$-semi-Tambara functor is a product-preserving functor $\Poly_G \to \Set$.
\end{definition}

A $G$-semi-Tambara functor comes equipped with certain operations on its output sets:

\begin{proposition}
    Let $X$ be a finite $G$-set, and let $F : \Poly_G \to \Set$ be a $G$-semi-Tambara functor. Let $\nabla_X : X \amalg X \to X$ denote the codiagonal. We obtain operations
    \begin{align*}
        +_{F,X}     & := F(X) \times F(X) \cong F(X \amalg X) \xrightarrow{F(T_{\nabla_X})} F(X) \\
        \cdot_{F,X} & := F(X) \times F(X) \cong F(X \amalg X) \xrightarrow{F(N_{\nabla_X})} F(X) \\
    \end{align*}
    which make $(F(X), +_{F,X}, \cdot_{F,X})$ into a commutative semiring\footnote{A semiring is just like a ring, except that it need not admit additive inverses of its elements. These are also known as \emph{rigs}.}.
\end{proposition}

A Tambara functor is simply a semi-Tambara functor for which all of these commutative semirings are in fact rings.

\begin{definition}
    A $G$-Tambara functor is a $G$-semi-Tambara functor $F$ such that $(F(X), +_{F,X})$ is an abelian group for all finite $G$-sets $X$.
\end{definition}

Since Tambara functors preserve products and $\varnothing$ is the terminal object in $\Poly_G$, every Tambara functor sends $\varnothing$ to a singleton set (whose commutative ring structure must be that of the zero ring). Moreover, every morphism in the category of finite $G$-sets decomposes canonically as a disjoint union of morphisms between transitive $G$-sets (or from the empty set to a transitive $G$-set). Every transitive $G$-set is isomorphic to $G/H$ for some subgroup $H \leq G$, and a morphism $G/H \to G/K$ is always a composition
\[G/H \xrightarrow{xH \mapsto xg^{-1}(gHg^{-1})} G/gHg^{-1} \to G/K\]
for some $g \in G$ such that $gHg^{-1} \subseteq K$. Thus, the data of a $G$-Tambara functor can also be specified by finitely many commutative rings with finitely many functions between them:

\begin{proposition}\label{prop:alt-def-of-Tambara-functors}
    A $G$-Tambara functor $F$ is determined (up to isomorphism) by the commutative rings $F(G/H)$ for all subgroups $H \leq G$, together with the operations
    \begin{align*}
        \tr_H^K     & := F(G/H) \xrightarrow{F(T_{G/H \to G/K})} F(G/K)               &  & \text{``restriction''} \\
        \nm_H^K     & := F(G/H) \xrightarrow{F(N_{G/H \to G/K})} F(G/K)               &  & \text{``transfer''}    \\
        \res_H^K    & := F(G/K) \xrightarrow{F(R_{G/H \to G/K})} F(G/H)               &  & \text{``norm''}        \\
        \conj_{H,g} & := F(G/H) \xrightarrow{F(T_{G/H \to G/gHg^{-1}})} F(G/gHg^{-1}) &  & \text{``conjugation''}
    \end{align*}
    for all subgroups $H \leq K \leq G$ and all elements $g \in G$.
\end{proposition}

This can be used to give an alternative definition of Tambara functors; see \cite{Str12} for details. In particular, we note that the 
$\conj_{H,g}$ operations give an action of the ``Weyl group,'' 
$W_G(H) := N_G(H)/H$ on $F(G/H)$, called the \emph{Weyl action}. 
The $\res_H^K$ operations are ring homomorphisms, 
the $\tr_H^K$ operations are $F(G/K)$-module homomorphisms,
and the $\nm_H^K$ operations are homomorphisms of underlying 
multiplicative monoids.

Lastly, a common source of Tambara functors is the fixed-point Tambara functor construction.
\begin{definition}
\label{defn:fixed-pt-tambara-functor}
If $R$ is a commutative ring with $G$ action, then the fixed point
Tambara functor of $R$, $\underline{R}$,
is the $G$-Tambara functor defined by $\underline{R}(G/H)= R^H$,
$c_{H,g} x = gx$, and when $H\subseteq K$, $\res^K_H x = x$, 
\[ \tr_H^K x = \sum_{kH\in K/H} kx
\text{, and }
\nm_H^K x = \prod_{kH\in K/H} kx.
\]
\end{definition}
When $R$ is an ordinary ring we regard it as having the trivial $G$-action,
in which case $\underline{R}$ is also called the \emph{constant Tambara functor
on $R$}.

\begin{lemma}
    \label{lem:fp-tambara-functor-adjt-to-eval-at-underlying}
    The fixed-point Tambara functor construction is right adjoint 
    to the evaluation functor $\ev_{G/e}:\GTamb\to \GCRing$ that 
    sends a Tambara functor $T$ to the ring $T(G/e)$ with its 
    Weyl group action.
\end{lemma}

\begin{proof}
    Let $S$ be a commutative ring with $G$-action and $T$ be a 
    $G$-Tambara functor. We'll describe the unit and counit of the 
    adjunction.

    The unit map,
    $\eta_T:T\to \underline{T(G/e)}$, is given componentwise by the 
    restriction maps
    $\res^H_e : T(G/H)\to T(G/e)^H$. That these maps are compatible 
    with the norms and transfers follows from the rules for composition.
    The counit map is even more straightforward, since $\ev_{G/e}(\underline{S})=S^e=S$, and the counit map will just be the identity map.

    The triangle identities follow from the facts that 
    $\eta_{\underline{S}}$ reduces to the identity map $\underline{S}\to \underline{S}$
    and that the $G/e$ component of $\eta_T$ is the identity map 
    $T(G/e)\to T(G/e)$.
\end{proof}

\subsection{The Burnside Tambara functor and the free Green functor}

The Burnside Tambara functor, which we will denote by $\A$,
is the monoidal unit for the box product of Mackey functors,
and therefore the initial Tambara functor, 
so for Tambara functors it plays 
the role that $\Z$ does for ordinary commutative rings.

Although the description above determines $\A$, we'll 
give it a more concrete definition.
\begin{definition}
    Let $\Poly_G^+$ denote the additive completion of the semi-additive 
    category $\Poly_G$. Then $\A$ is the representable functor 
    $\A(X) := \Poly_G^+(\varnothing, X)$. 
\end{definition}

Since representable functors are product preserving,
$\A$ is a Tambara functor,
and by the Yoneda lemma, $\GTamb(\A, R)\cong R(\varnothing)\cong *$,
since $\varnothing$ is the terminal object in $\Poly_G$ and $R$ is
product preserving, 
where we use $*$ to denote a one element set or $G$-set. Therefore 
$\A$ is, in fact, the initial Tambara functor, so this agrees with our 
description above.

Then for an orbit $G/H$, we have that $\A(G/H)$ is
the group completion of the monoid of 
isomorphism classes of diagrams of $G$-sets of the form
\[\varnothing 
\leftarrow 
\varnothing 
\rightarrow X 
\toby{p} 
G/H.
\]
The category of $G$-sets over $G/H$ is equivalent to the 
category of $H$-sets, so $\A(G/H)$ is naturally isomorphic to the 
Burnside ring of $H$, which is the Grothendieck ring of the category 
of finite $H$-sets.

Now a Green functor is a monoid for the box product of 
$G$-Mackey functors. Green functors can equivalently be described 
as a Mackey functor,
$R$, with a ring structure on every abelian group $R(X)$ such 
that the restriction maps are ring homomorphisms and the 
transfers are module homomorphisms. 

Now for any Green functor $R$ and ordinary ring $A$, let 
$R\otimes A$ denotes the levelwise tensor product of 
$R$ with $A$, $(R\otimes A)(X) := R(X)\otimes A$.
$R\otimes A$ is naturally a Green functor as well,
since the tensor product of 
a ring homomorphism or module homomorphism with 
a fixed ring is still a ring or module homomorphism. 

\begin{lemma}
    If $R$ is a Green functor and $A$ is a commutative ring, then
    the levelwise tensor product $R\otimes A$ represents the 
    functor that sends a Green functor $S$ to pairs $(f,g)$
    of a Green functor morphism $f:R\to S$ and a ring morphism 
    $g:A\to S(*)$. The universal pair consists of 
    the Green functor 
    map $\iota_R : R\to R\otimes A$ which sends $r$ to $r\otimes 1$ and 
    the ring map $\iota_A: A\to R(*)\otimes A$ which sends $a$ to $1\otimes a$.
\end{lemma}

\begin{proof}
    If $S$ is a Green functor, then a map $\phi: R\otimes A\to S$
    has ring maps $\phi_X : R(X)\otimes A\to S(X)$ as components.

    The component of our natural transformation 
    corresponding to the pair 
    $(\iota_R,\iota_A)$ is the map
    $\GGreen(R\otimes A, S)\to \GGreen(R,S)\otimes \Ring(A, S(*))$
    that sends $\phi: R\otimes S$ to the pair 
    $(\phi\iota_R, \phi_*\iota_A)$.

    To show that this is an isomorphism, we'll show that this map 
    is invertible. So suppose we're given $f:R\to S$ and 
    $g:A\to S(*)$. Then if we let $t_X:X\to *$ be the terminal 
    map and $t_X^* : S(*)\to S(X)$ be the restriction, 
    then we can define 
    $\psi_X : R(X)\otimes A \to S(X)$ to be 
    \[ 
    R(X)\otimes A\toby{f_X\otimes t_X^*g} S(X)\otimes S(X) \to S(X).
    \]

    Then for any map $\alpha:X\to Y$, if $\alpha^*$ is the restriction map 
    along $\alpha$, then 
    \[
    \alpha^* \psi_Y(r\otimes a) = 
    (\alpha^* f_Y(r))(\alpha^*t_Y^*g(a))  
    = (f_X(\alpha^*r))(t_X^* g(a))
    = \psi_X(\alpha^*(r\otimes a)).
    \]
    Similarly, if $T_{\alpha}$ represents the transfer map along 
    $\alpha$, then 
    \[
    T_{\alpha}\psi_X(r\otimes a)
    =T_{\alpha} (f_X(r) t_X^* g(a))
    =T_{\alpha} (f_X(r) \alpha^*t_Y^* g(a))
    = (T_{\alpha} f_X(r)) t_Y^*g(a)
    \]
    \[
    = f_Y(T_{\alpha} r)t_Y^*g(a)
    = \psi_Y(T_{\alpha}(r\otimes a)).
    \]
    So $\psi$ is a valid map of Green functors.

    Then $\psi\iota_R(r) = \psi(r\otimes 1) = f(r)$, and 
    $\psi_*\iota_A(a) = \psi_*(1\otimes a) = g(a)$. 
    Conversely, if we start with a map $\phi$ and build $\psi$
    from $f=\phi\iota_R$ and $g=\phi_*\iota_A$, then 
    \[ 
    \psi_X(r\otimes a) = \of*{\phi_X\iota_R(r)}\of*{t_X^*\phi_*\iota_A(a)}
    = \phi_X(r\otimes 1)\phi_X(1\otimes a)
    = \phi_X(r\otimes a).
    \]
\end{proof}

\begin{corollary}
    \label{cor:levelwise-tensor-is-adjt-to-evaluation-at-*-for-green}
    The functor $A\mapsto R\otimes A$ is the left adjoint 
    to the evaluation functor $\ev_*^{\text{Green}} : \Alg[R] \to \CRing$ that
    sends an $R$-algebra in Green functors, $S$, to $S(*)$.
\end{corollary}

Since we can write that evaluation functor as a composite of 
the forgetful functor from $R$-algebras to Green functors 
and then as evaluation at $*$ from Green functors to commutative 
rings, we have that the left adjoint of the composite is 
naturally isomorphic to the composite of the left adjoints of 
each of those functors, which gives us the following additional
corollary.

\begin{corollary}
    There is a natural isomorphism $R\otimes A\cong R\boxtimes (\A\otimes A)$, where $\A$ is the Burnside Tambara functor.
\end{corollary}

\begin{proof}
    By the above corollary, since $\A$ is initial in Green functors,
    so the category of Green functors is equivalent to the 
    category of 
    $\A$-algebras, the adjoint to evaluation at $*$ from Green
    functors is $A\mapsto \A \otimes A$. Then since $\boxtimes$
    is the coproduct of Green functors, $S\mapsto R\boxtimes S$
    is left adjoint to the forgetful functor from $R$-algebras
    to Green functors.
\end{proof}

Now we're particularly interested in the case where 
$A$ is a polynomial ring, $\Z[x_1,\ldots,x_n]$. By 
\Cref{cor:levelwise-tensor-is-adjt-to-evaluation-at-*-for-green},
we have that $\A[\underline{x}] := \A \otimes \Z[x]$ is the free 
Green functor on a single element adjoined at the $G/G$ level.
Our goal is to show that in fact $\A[\underline{x}]$ has a 
Tambara functor structure, which is characterized by the 
property that 
$\nm_K^H \res^G_K \underline{x}_i = \res^G_H\underline{x}_i^{[H:K]}$
for $K\subseteq H \subseteq G$. 

If that holds, then when $R$ is a Tambara functor, then 
if we define $R[\underline{x}_1,\ldots,\underline{x}_n]$ 
to be 
$R\boxtimes \A[\underline{x}_1]\boxtimes \cdots \A[\underline{x}_n]
\cong R\otimes \Z[x_1,\ldots,x_n]$, 
then $R[\underline{x}_1,\ldots,\underline{x}_n]$ also has 
a Tambara functor structure, and it will be characterized by the 
property that 
$\nm_K^H \res^G_K \underline{x}_i = \res^G_H\underline{x}_i^{[H:K]}$
for $K\subseteq H \subseteq G$ and $1\le i \le n$.

In order to prove this, we'll need results of Nakaoka from 
\cite{Nak11}. In particular, the results we're using are 
Proposition 2.11, 
Theorem 2.12, Corollary 2.14, Theorem 2.15 and Corollary 2.17 of that
paper, which can be summarized in the following way:
\begin{theorem}[Nakaoka]
    If $M$ is a semi-Mackey functor, thought of as having
    restrictions and norms, then there is a 
    free Tambara functor generated by $M$, $\A[M]$, 
    whose elements over a given $G$-set $X$ are isomorphism classes
    of pairs 
    $(\alpha : S\to X, m\in M(S))$, which can be thought of 
    as formal transfers $T_{\alpha}m$. Here free means that 
    the functor $M\mapsto \A[M]$ is 
    left adjoint to the forgetful functor 
    from Tambara functors to their multiplicative semi-Mackey 
    functors. Moreover, the same is true 
    if you choose a subcategory of allowed norms and work with the 
    corresponding 
    incomplete semi-Mackey functors and incomplete Tambara functors.
\end{theorem}

\begin{lemma}
If $M$ is an ordinary monoid, and $\underline{M}$ is the 
constant semi-Mackey functor on $M$, then $\A[\underline{M}]$ 
is isomorphic as a Green functor to $\A\otimes \Z[M]$.
\end{lemma}

\begin{proof}
The key is to show that there is a natural isomorphism
$\GGreen(\A[\underline{M}],S) \cong \mathsf{Mon}(M, S(*))$. Then 
we have the sequence of natural isomorphisms
\[ 
\GGreen(\A[\underline{M}], S)
\cong \mathsf{Mon}(M, S(*))
\cong \CRing(\Z[M], S(*))
\cong \GGreen(\A\otimes \Z[M], S),
\]
which shows $\A[\underline{M}]\cong \A\otimes \Z[M]$.

Nakaoka's theorem constructs the same object regardless of 
what subcategory of allowed norms you take, so
if we take the minimal subcategory generated by the fold maps,
then in fact, Nakaoka's theorem tells us 
\[ 
\GGreen(\A[\underline{M}],S)\cong \mathsf{Coeff}_{\mathsf{Mon}}(\underline{M}, S)\cong \mathsf{Mon}(M, S(*)),
\]
where $\mathsf{Coeff}_{\mathsf{Mon}}$ is the category 
of coefficient systems in monoids, which is to say product 
preserving functors from $\Gset{}^{\text{op}}$ to $\mathsf{Mon}$.
\end{proof}

\begin{corollary}
    $\A\otimes \Z[x]\cong \A[\underline{\N}]$,
    and therefore has a Tambara functor structure,
    with $\nm_K^H x = x^{[H:K]}$ for $K\subseteq H \subseteq G$.
\end{corollary}
\begin{proof}
    Write the elements of the monoid $\N$ as $t^n$ so that we can 
    distinguish them from ordinary uses of the integers.
    Then 
    since $\Z[x]\cong \Z[\N]$, with $x$ going to $t^1$,
    we have 
    \[
    \A\otimes \Z[x]\cong \A\otimes \Z[\N]\cong \A[\underline{\N}],
    \]
    and $\nm_K^H t^1 = t^{[H:K]}$ in $\underline{\N}$.
\end{proof}

Altogether, the results in this section establish the following 
result:

\begin{proposition}
    \label{prop:free-green-has-norms-and-is-levelwise-tensor}
    For a Tambara functor $R$, 
    $R[\underline{x}_1,\ldots,\underline{x}_n]$ 
    is an $R$-Tambara algebra with norms characterized by 
    $\nm_K^H \res^G_K \underline{x}_i = \res^G_H\underline{x}_i^{[H:K]}$,
    which is isomorphic as a Green functor to the levelwise tensor
    product,
    $R\otimes \Z[x_1,\ldots,x_n]$.
\end{proposition}

\subsection{Coinduction}

Let $H \leq G$ be an inclusion of finite groups. There is a \emph{restriction} functor $R^G_H : \GTamb \to \GTamb[H]$, given levelwise by $(R^G_H k)(H/K) = k(G/K)$. More precisely, $R^G_H$ is given by precomposition with the product-preserving functor $G \times_H {-} : \Poly_H \to \Poly_G$.

For completely formal reasons \cite[Section 18]{Str12}, $R^G_H$ has both a left adjoint, denoted $N^G_H$ and called \emph{induction} or \emph{norm}, and a right adjoint, denoted $C^G_H$ and called \emph{coinduction}. It turns out that $G \times_H {-} : \Poly_H \to \Poly_G$ is right adjoint to the forgetful functor $\Poly_G \to \Poly_H$, and it follows that $C_H^G$ is given by precomposition with the forgetful functor \cite[Theorem 6.4]{BH18}.

Recall that that the category of $e$-Tambara functors is naturally identified with the category of commutative rings, so the coinduction functor, $C_e^G$, can be described as going from commutative rings to $G$-Tambara functors. This functor plays
an important role in our classification of Nullstellensatzian Tambara
functors, so we need to say what it does more concretely. We do so by 
observing that the restriction functor 
$R^G_e:\GTamb\to\GTamb[e]\simeq \CRing$ is 
the composite of the evaluation functor, $\ev_{G/e}:\GTamb\to \GCRing$,
which remembers the Weyl action, and the forgetful functor, 
$U:\GCRing\to \CRing$, so $C_e^G$ is the composite of the 
right adjoints of these functors. By \Cref{lem:fp-tambara-functor-adjt-to-eval-at-underlying} at the end of 
the previous section, the adjoint to $\ev_{G/e}$ is given by the 
fixed-point Tambara functor construction, so we just need to 
determine the right adjoint to the forgetful functor. This is 
well-known in its own right, as well as a special 
case of a well-known result about the right adjoint to evaluation on 
presheaf categories, but we'll give the specific result and its proof here.

\begin{lemma}
    Given a commutative ring $R$, the right adjoint to the forgetful 
    functor, $U:\GCRing \to \CRing$, sends $R$ to 
    the $G$-ring $\Fun(G,R)$ of 
    functions $G \to R$, whose $G$-action is given by 
    $(g \cdot \varphi)(h) = \varphi(h g)$. 
\end{lemma}

\begin{proof}
    Let $S$ be a commutative ring with $G$-action, and let $R$ be an 
    ordinary commutative ring. 
    
    If $\varphi : S \to \Fun(G,R)$ is a map of commutative rings 
    with $G$ action, then the adjoint, $\tilde{\varphi}:US\to R$,
    is given by
    $\tilde{\varphi}(s) := \varphi_s(e)$. On the other hand, if we start with
    some ring map
    $\tilde{\varphi}: US\to R$, we can recover $\varphi$ from the fact that
    if $\varphi$ exists, it must satisfy
    \[\varphi_s(g) = \varphi_s(e\cdot g) = (g\cdot \varphi_s)(e) = \varphi_{gs}(e) = \tilde{\varphi}(gs),\]
    but that gives us a definition of $\varphi$.

\end{proof}

Putting the pieces together, we get a clear description of 
$C_e^G R$:
\begin{lemma}\label{lem:coinduction-from-e}
    Let $R$ be a commutative ring. Then $C_e^G R = \underline{\Fun(G,R)}$.
\end{lemma}
\begin{proof}
    We note that $R^G_e = k \mapsto k(G/e) : \GTamb \to \CRing$ factors through $\GCRing$, by first remembering the Weyl action on $k(G/e)$ and then forgetting it. Thus, we compute $C_e^G$ as a composition of two right adjoints, which are as follows:
\[\begin{tikzcd}[ampersand replacement=\&]
	{G{-}\Tamb} \\
	{G{-}\CRing} \\
	\CRing
	\arrow[""{name=0, anchor=center, inner sep=0}, "{\ev_{G/e}}"', shift right=2, from=1-1, to=2-1]
	\arrow[""{name=1, anchor=center, inner sep=0}, "{S \mapsto \underline{S}}"', shift right=2, from=2-1, to=1-1]
	\arrow[""{name=2, anchor=center, inner sep=0}, "{\text{forget}}"', shift right=2, from=2-1, to=3-1]
	\arrow[""{name=3, anchor=center, inner sep=0}, "{R \mapsto \Fun(G,R)}"', shift right=2, from=3-1, to=2-1]
	\arrow["\dashv"{anchor=center}, draw=none, from=0, to=1]
	\arrow["\dashv"{anchor=center}, draw=none, from=2, to=3]
\end{tikzcd}\]

\end{proof}

\section{Recollections on the Commutative Algebra of Tambara Functors}

In this section we review the established story of commutative algebra in Tambara functors. Nakaoka \cite{Nak11a} \cite{Nak11b} has established relatively well-behaved notions of prime ideals, localizations, integral domains, and fields for Tambara functors. On the other hand, Blumberg and Hill \cite{BH18} established a useful notion of ``free polynomial algebra'' which will play a central role in our study of Nullstellensatzian Tambara functors. For an introduction to $G$-Tambara functors, see \cite{Str12} and \cite{BH18}.

\subsection{Prime Ideals and Tambara Fields}

Nakaoka originally defined prime ideals and field-like Tambara functors in \cite{Nak11a}. We collect some of the relevant results here, and refer the interested reader to the excellent exposition in \cite{Nak11a}.

\begin{definition}
    An \emph{ideal} of a Tambara functor $k$ is a kernel of a morphism from $k$. More precisely, an ideal $I$ of $k$ is a collection of subsets $I(G/H) \subseteq k(G/H)$ (for all subgroups $H \leq G$) such that there is some morphism $\varphi : k \to k'$ such that $I(G/H) = \{x \in k(G/H) : \varphi_{G/H}(x) = 0\}$ for all $H \leq G$.

    Equivalently, a collection of subsets $I(G/H) \subseteq k(G/H)$ forms an ideal if and only if the following conditions all hold:
    \begin{enumerate}
        \item $I(G/H)$ is an ideal of the commutative ring $k(G/H)$ for all $H \leq G$;
        \item $I(G/H)$ is closed under the Weyl action on $k(G/H)$ for all $H \leq G$;
        \item $I$ is closed under transfer, norm, and restriction for all inclusions of subgroups $H' \leq H \leq G$.
    \end{enumerate}
\end{definition}





\begin{definition}
    A Tambara functor $k$ is said to be \emph{field-like} (or a \emph{Tambara field}) if $0$ is the unique proper ideal of $k$.
\end{definition}

The following gives a relatively straightforward and checkable condition for when a Tambara functor is field-like.

\begin{proposition}[{\cite[Theorem 4.32]{Nak11a}}]
    A Tambara functor $k$ is field-like if and only if all restriction maps $k(G/e) \rightarrow k(G/H)$ are injective and $k(G/e)$ has no nontrivial $G$-invariant ideal.
\end{proposition}

In particular, if $k$ is a Tambara functor with each $k(G/H)$ a field, then $k$ is field-like. So, given a field $\mathbb{F}$, the constant Tambara functor at $\mathbb{F}$ is field-like. Additionally, if $\mathbb{F} \subset L$ is a Galois extension with Galois group $G$, then the fixed-point Tambara functor of $L$ is field-like. One may also show that the coinduction of any field-like $H$-Tambara functor to a $G$-Tambara functor is field-like. In fact, the converse is also true.

\begin{theorem}[\cite{Wis24} \cite{Wis25}]
    Let $k$ be a field-like $G$-Tambara functor. Then $k \cong \mathrm{C}_H^G \ell$ for some field-like $H$-Tambara functor $\ell$ such that each $\ell(G/H)$ is a field. Let $R \cong \mathrm{C}_H^G S$ be any isomorphism of $G$-Tambara functors. Then $R$ is field-like if and only if $S$ is.
\end{theorem}

However, there are field-like Tambara functors which are not fixed-point Tambara functors (in particular, not coinduced, not constant, and not arising from a Galois extension). For example (cf. \cite{Wis24}), let $\ell$ be a non-perfect field of characteristic $p$, and define a $C_p$-Tambara functor $k_\ell$ by setting $k_\ell(C_p/e)$ equal to $\ell$ with the trivial $C_p$ action and $k_\ell(C_p/C_p)$ equal to the image of the Frobenius endomorphism. The restriction map is the inclusion of the image of Frobenius (which is a proper subfield since $\ell$ is not perfect), the transfer map is zero, and the norm map is Frobenius. If $G$ is any nontrivial finite group, then it has a subgroup $C_p$ for some prime $p$. Coinducing this example to $G$ then produces a field which is not a fixed-point Tambara field.

\subsection{Free Algebras}

\begin{proposition}\label{prop:free-forget-kalg}
    Let $k$ be a Tambara functor. There is a free-forgetful adjunction between Tambara functors and $k$-algebras. The forgetful functor sends a $k$-algebra $k \rightarrow K$ to $K$, and the free functor is given by the box product $k \boxtimes -$.
\end{proposition}
\begin{proof}
    This comes from the fact that $\boxtimes$ is the coproduct in the category of Tambara functors. Given a Tambara functor $T$ and a $k$-algebra $\alpha : k \to K$, a morphism of $k$-algebras $k \boxtimes T \to K$ is a morphism $f : k \boxtimes T \to K$ of Tambara functors making the diagram
    \[\begin{tikzcd}[ampersand replacement=\&]
            T \&\& {k \boxtimes T} \&\& K \\
            \\
            \&\& k
            \arrow["i", from=3-3, to=1-3]
            \arrow["\alpha"', from=3-3, to=1-5]
            \arrow["f", dashed, from=1-3, to=1-5]
            \arrow["j", from=1-1, to=1-3]
        \end{tikzcd}\]
    commute, where $i$ and $j$ are the structure maps of the coproduct
    $k \boxtimes T$. By the universal property of coproducts, $f$ is
    uniquely determined by the data of a pair of morphisms
    $(f_1 : k \to K, f_2 : T \to K)$ such that $f_1 \circ i = \alpha$.
    Sending $f$ to $f_2 \circ j$ gives a natural bijection from
    $\kAlg(k \boxtimes T, K)$ to $\GTamb(T,K)$.
\end{proof}

\begin{proposition}[{\cite[Definition 5.4]{BH18}}]
    For any subgroup $H$ of $G$, there exists a Tambara functor $\mathcal{A}[x_H]$ which represents the functor $K \mapsto K(G/H)$ sending a Tambara functor to its level $H$ ring.
\end{proposition}

\begin{corollary}
    \label{cor:free-kalg-corepresents}
    Let $k$ be a Tambara functor. In the category of $k$-algebras, there exists a Tambara functor $k[x_H]$ which represents the functor $K \mapsto K(G/H)$ sending a $k$-algebra to its level $H$ ring.
\end{corollary}

\begin{proof}
    In fact, $k[x_H]$ is given by $k \boxtimes \A[x_H]$. For any $k$-algebra $K$, we have
    \begin{align*}
        \kAlg(k \boxtimes \A[x_H], K) & \cong \GTamb(\A[x_H], K) \\
                                       & \cong K(G/H).
    \end{align*}
    This isomorphism is natural in $K$, so this verifies the desired universal property.
\end{proof}

\begin{lemma}[{\cite[Proposition 5.2]{BH18}}]
    In the category of Tambara functors, filtered colimits may be computed levelwise. More precisely, we have
    \[ \left( \colim_i K_i \right)(G/H) \cong \colim_i \left( K_i(G/H) \right)\]
    for filtered diagrams $K : I \to \mathsf{Tamb}$. In other words, $\A[x_H]$ is compact in the category of Tambara functors.
\end{lemma}

\begin{corollary}
    \label{cor:filtered-colimits-computed-levelwise}
    Let $k$ be a Tambara functor. In the category of $k$-algebras, filtered colimits may be computed levelwise.
\end{corollary}

\begin{proof}
    By \Cref{prop:free-forget-kalg}, we have an isomorphism
    \begin{align*}
        \kAlg(k[x_H], \colim K) & = \kAlg(k \boxtimes \A[x_H], \colim K) \\
                                 & \cong \GTamb(\A[x_H], \colim K)      \\
                                 & \cong \colim_i \GTamb(\A[x_H], K_i)  \\
                                 & \cong \colim_i \kAlg(k[x_H], K_i)
    \end{align*}
    for filtered diagrams $K : I \to \kAlg$.
\end{proof}

\begin{definition}
    The free polynomial $k$-algebra on $n$ generators in respective levels $H_1,...,H_n$ is \[ k[x_{H_1},...,x_{H_n}] := k \boxtimes \mathcal{A}[x_{H_1}] \boxtimes ... \boxtimes \mathcal{A}[x_{H_n}] \]
\end{definition}

Repeating the argument of the previous proof, we see that the free polynomial $k$-algebra $k[x_{H_i}]_{i \in I}$ represents the functor sending a $k$-algebra $K$ to the set $\prod_{i \in I} K(G/H_i)$.

\begin{lemma}\label{lem:res of free on fixed}
    For any $G$-Tambara functor $k$, $R^G_H(k[x_G]) \cong (R^G_H k)[x_H]$.
\end{lemma}
\begin{proof}
    We start with the case $k = \A$. Because coinduction is given by precomposition with restriction, we have
    \begin{align*}
        \GTamb[H](R^G_H \A[x_G], T) & \cong \GTamb(\A[x_G], C_H^G T) \\
                                  & \cong (C_H^G T)(G/G)            \\
                                  & \cong T(H/H)                    \\
                                  & \cong \GTamb[H](\A[x_H], T)
    \end{align*}
    naturally in $T$. Thus, $R^G_H \A[x_G] \cong \A[x_H]$.

    Next, since $R^G_H$ is a left adjoint, it preserves coproducts. Thus,
    \[R^G_H k[x_H] \cong R^G_H (k \boxtimes \A[x_G]) \cong R^G_H k \boxtimes R^G_H \A[x_G] \cong (R^G_H k)[x_H].\qedhere\]
\end{proof}

We record one more fact here, which will eventually allow us to prove that certain Tambara functors are finitely presented. 

\begin{proposition}
    \label{prop:free-at-fixed-pts-is-levelwise-fg}
  For all $H\le G$,
  $\mathcal{A}_G[x_G](G/H)$ is a finitely generated ring. 
\end{proposition}

In order to prove the proposition we first need a lemma.

\begin{lemma}
  If $R$ is a Tambara functor, then for $K\le H$, 
  $\res^H_K : R(G/H)\to R(G/K)$ is an integral map 
  of rings.
\end{lemma}

\begin{proof}
  Consider the Tambara functor 
  $S:=R[\underline{x}]$, which is isomorphic to the levelwise
  tensor product by \Cref{prop:free-green-has-norms-and-is-levelwise-tensor},
  so we have that $S(G/H)\cong R(G/H)[x]$. 
  
  Then if $a\in R(G/K)$, let $p(x) = \nm_K^H (x-a) 
  \in R(G/H)[x]$. The 
  Mackey double coset formula tells us that $(x-a)$ 
  divides $\res^H_K\nm_K^H (x-a)$, and 
  $p(a) = \res^H_K(p)(a)$, so $p(a)=0$.
  Since $p$ is monic 
  (its leading term is $N_K^H x = x^{[H:K]}$),
  $a$ is integral over $R(G/H)$.
\end{proof}

Now we can prove the proposition.

\begin{proof}[Proof of \Cref{prop:free-at-fixed-pts-is-levelwise-fg}]
  We induct on the order of $G$. When $G=e$ is the trivial group, 
  $\mathcal{A}_e[x_e] \cong \Z[x]$, which is certainly finitely generated.
  Otherwise, assume that the 
  result holds for all subgroups of $G$. Note that 
  $R^G_H \mathcal{A}_G[x_G] = \mathcal{A}_H[x_H]$, so this tells us that all 
  of the rings $\mathcal{A}_G[x_G](G/H)$ are finitely 
  generated for $G\ne H$.

  Now for each $H < G$ choose generators $y^H_i$ 
  of $\mathcal{A}_G[x_G](G/H)$, and for later convenience 
  assume that $y^H_0 = \res^G_H x_G$ for each subgroup 
  $H$. For each $y^H_i$, by integrality of restriction maps we can choose a monic polynomial $m^H_i(s) \in \mathcal{A}_G[x_G][s]$
  such that $m^H_i(y^H_i)=0$.

  Let $S$ be the subring 
  of $\mathcal{A}_G[x_G](G/G)$ generated by $\mathcal{A}_G(G/G)$, 
  the coefficients of the $m^H_i$, and the elements 
  $\nm_H^G y^H_i$. This is a finitely generated ring,
  since $\mathcal{A}_G(G/G)$ is finitely generated, and there 
  are only finitely many elements $y^H_i$. 
  Moreover, by construction each of the rings 
  $\mathcal{A}_G[x_G](G/H)$ is still integral over $S$, since 
  $S$ contains the coefficients of monic polynomials 
  killing the generators of those rings. Therefore,
  since the rings $\mathcal{A}_G[x_G](G/H)$ are finitely
  generated and integral over $S$, they are module finite 
  over $S$. So for each proper subgroup $H$ we can choose 
  finitely many elements $z^H_j$ which generate 
  $\mathcal{A}_G[x_G](G/H)$ as a module over $S$.

  Then $S$ and the elements $\tr_H^G z^H_j$ 
  together generate all of $\mathcal{A}_G[x_G](G/G)$.
  This is because $\mathcal{A}_G[x_G](G/G)$ is generated
  over $\mathcal{A}_G(G/G)$ by 
  elements of the form 
  \[\tr_H^G \prod_i \nm_{K_i}^H \res^G_{K_i} x_G,\]
  for some choice of $K_i\le H\le G$. 
  Now if $H=G$, all of the elements 
  $\nm_{K_i}^G\res^G_{K_i} x_G$ are in $S$ already,
  since $\res^G_{K_i}x_G = y^{K_i}_0$ and $S$ includes
  the norms of our chosen generators. 
  
  Otherwise, when $H\ne G$, the generator is of the 
  form $\tr_H^G a$ for some element $a$ of 
  $\mathcal{A}_G[x_G](G/H)$, and we can write $a$ in the form 
  $\sum_{k=1}^n \res^G_H(a_k)z^H_k$ for elements 
  $a_k\in S$.
  But then 
  \[
    \tr_H^G a = \tr_H^G \sum_{k=1}^n a_k\cdot z^H_k
    = \sum_{k=1}^n a_k \tr_H^G z^H_k,
  \]
  which is in our ring.
\end{proof}

\section{Adjunctions and Compact Algebras}

Since Nullstellensatzian Tambara functors are defined in terms of compact objects, we must establish some basic facts about these.

\begin{definition}
    An object $x$ in a category $\catC$ is said to be \emph{compact} if the functor $\catC(x,-)$ preserves filtered colimits.
\end{definition}

For categories of algebraic objects (more precisely, for multi-sorted varieties in the sense of universal algebra), compact objects are precisely those objects which are finitely presented \cite[Corollary 3.13]{Adamek_Rosicky_1994}. For example, if $R$ is a commutative ring, the compact objects of the category of $R$-algebras are precisely the finitely presented $R$-algebras. More generally, if $k$ is a Tambara functor, the compact objects of the category of $k$-algebras are precisely the finitely presented $k$-algebras.

\begin{definition}
    Let $k$ be a Tambara functor. A $k$-algebra is finitely generated if it is a quotient of some $k[x_{H_1}, \dots, x_{H_n}]$, and finitely presented if furthermore the quotient is by a finitely generated Nakaoka ideal. Equivalently, a finitely presented $k$-algebra is the coequalizer of a diagram \[ k[x_{H_i}]_{i \in I} \rightrightarrows k[x_{H_j}]_{j \in J}\] with $I$ and $J$ finite sets.
\end{definition}

\begin{proposition}\label{prop:compact same as fp}
    Let $k$ be a Tambara functor. A $k$-algebra is compact if and only if it is finitely presented.
\end{proposition}
\begin{proof}
    This follows immediately from \cite[Corollary 3.13]{Adamek_Rosicky_1994} once we establish that the category of $k$-algebras is a multi-sorted variety. So, we will encode the category of $k$-algebras as the category of algebras for a multi-sorted equational theory.

    We have one sort for each subgroup of $G$, so that a $k$-algebra $X$ consists of sets $\{X(G/H)\}_{H \leq G}$. We have a nullary operation $x : * \to X(G/H)$ for each $x \in k(G/H)$, encoding the structure map $k \to X$. We additionally have unary operations $\res, \nm, \tr, \conj$ to encode the Tambara operations of $X$, and nullary, unary, and binary operations at each level of $X$ to encode the commutative ring structures of the levels of $X$. Then we impose equational axioms which encode the axioms of a Tambara functor, as well as the fact that $k \to X$ is a homomorphism of Tambara functors. Algebras for this equational theory are precisely $k$-algebras.
\end{proof}








\subsection{Nullstellensatzian Objects}

The authors of \cite{BSY22} introduce the notion of a Nullstellensatzian object, which generalizes the notion of an algebraically closed field. The idea is that a Nullstellensatzian object ``satisfies the conclusion of Hilbert's Nullstellensatz,'' and for commutative rings this turns out to be equivalent to being an algebraically closed field.

\begin{definition}
    Let $\catC$ be a category with an initial object. $\catC$ is \emph{Nullstellensatzian} if all compact
    nonterminal objects of $\catC$ admit a map to the initial object.
\end{definition}

\begin{definition}
    Let $\catC$ be a category. A morphism $f : x \rightarrow y$ in $\catC$ is called compact if it is a compact object of the coslice category $x / \catC$ of objects under $x$.
\end{definition}

\begin{definition}
    Let $\catC$ be a category admitting a terminal object. An object $x$ of $\catC$ is said to be \emph{Nullstellensatzian} if both:
    \begin{enumerate}
        \item $x$ is not terminal;
        \item For all compact morphisms $f : x \to y$ with $y$ nonterminal, there exists a morphism $g : y \to x$ such that $g \circ f = \id_x$.
    \end{enumerate}

    Note that the second condition is equivalent to $x/\catC$ being a
    Nullstellensatzian category.
\end{definition}

\begin{proposition}
    The Nullstellensatzian objects of the category of commutative rings are precisely the algebraically closed fields.
\end{proposition}

\begin{proof}
    This is shown in \cite{BSY22}, Theorem 6.1. It is essentially an immediate consequence of Hilbert's Nullstellensatz.
\end{proof}

\subsection{Sufficient Conditions for a Functor to Preserve Nullstellensatzian Objects}

We are interested in studying Nullstellensatzian objects, and in this subsection we give one method which sometimes constructs them: \cref{lem:right-adjoints-transfer-AC-if-left-reflects-terminal-objects}, which roughly says that right adjoints which preserve filtered colimits often produce Nullstellensatzian objects.

\begin{lemma}\label{lem:left-adjoint-preserves-compacts}
    Let $F : \catC \to \mathsf{D}$ be left adjoint to some functor $G$ which preserves filtered colimits. Then $F$ sends compact objects to compact objects.
\end{lemma}
\begin{proof}
    Let $x$ be a compact object of $\catC$, and let $Y : I \to \mathsf{D}$ be a filtered diagram in $\mathsf{D}$. We have
    \begin{align*}
        \catD(Fx, \colim Y) & \cong \catC(x, G \colim Y)           \\
                                        & \cong \catC(x, \colim (G \circ Y))   \\
                                        & \cong \colim_i \catC(x, G Y_i)       \\
                                        & \cong \colim_i \catD(F x, Y_i),
    \end{align*}
    so $Fx$ is a compact object of $\mathsf{D}$.
\end{proof}

\begin{lemma}\label{lem:pres-filt-colims}
    Let $\catC$ be a category, and let $f : x \to y$ be a morphism in $\catC$. Then
    \begin{enumerate}
        \item $\cod : x / \catC \to \catC$ preserves filtered colimits;
        \item $f^* : y / \catC \to x / \catC$ preserves filtered colimits.
    \end{enumerate}
\end{lemma}
\begin{proof}
    The first claim is \cite[Corollary 1.4]{osmond_coslices_2021}. The second claim reduces to the first: notice that $f^* : y/\catC \to x/\catC$ is also given by
    \[y/\catC \cong f/(x/\catC) \xrightarrow{\cod} x/\catC.\qedhere\]
\end{proof}

\begin{lemma}\label{lem:colimits-in-under-categories}
    Let $\catC$ be a category and $x$ an object in $\catC$. If $I$ is a category, then the categories of diagrams $\alpha$ of shape $I$
    in $x/\catC$ and of diagrams $\alpha_+$ of shape $I_+$ in $\catC$ are equivalent, where $\alpha_+(0)=x$ and $I_+$ is $I$ with a freely adjoined initial object $0$. Moreover, for any diagram $\alpha: I\to x/\catC$ the inclusion of $I^+$ into $I^+_+$ induces
    an equivalence of categories from cocones on $\alpha_+$ to cocones
    on $\alpha$, where $-^+$ is the operation of freely adjoining a
    terminal object to a category. Therefore $\alpha$ has a colimit
    if and only if $\alpha_+$ does.
\end{lemma}

\begin{proof}
    Consider the diagram
    \[
        \begin{tikzcd}
            \mathsf{Cocones}(\alpha) \arrow[r] \arrow[d] & {\Cat(I^+,x/\catC)} \arrow[d] \arrow[r, hook] & {\Cat(I^+_+,\catC)} \arrow[d]        \\
            * \arrow[r, "\alpha"]                        & {\Cat(I,x/\catC)} \arrow[d] \arrow[r, hook]   & {\Cat(I_+,\catC)} \arrow[d, "\ev_0"] \\
            & * \arrow[r, hook]                             & \catC
        \end{tikzcd}
    \]
    All of the squares are pullback squares. The lower right square being
    a pullback square proves the first claim, and then the fact that
    the upper left square is a pullback square (by definition of the
    category of cocones to $\alpha$)
    and that the upper right is a pullback square
    tell us that the upper rectangle is also a pullback square,
    so the category of cocones to $\alpha$ is equivalent to the category
    of cocones to $\alpha_+$.
\end{proof}

\begin{corollary}\label{cor:if-g-preserves-filt-colim-then-so-do-slices-of-g}
    If $G:\catD\to \catC$ preserves filtered colimits, then for any
    $y\in \catD$, $y/G:y/\catD\to Gy/\catC$ preserves filtered colimits.
\end{corollary}
\begin{proof}
    Suppose that we have a diagram
    $\alpha:I\to y/\mathsf{D}$ where $I$ is a filtered category
    and suppose this has a colimit given by an extension
    $\beta:I^+\to y/\mathsf{D}$.
    Then by \Cref{lem:colimits-in-under-categories} above,
    since $\alpha$ has a colimit the corresponding diagram
    $\alpha_+$ has a colimit which will be given by $\beta_+$.
    If $I$ is filtered then
    $I_+$ is also filtered, so $G\circ \beta_+$ gives the colimit of
    $G\circ \alpha_+$ since $G$ preserves filtered colmits.
    And
    $G\circ\beta_+\cong ((y/G)\circ \beta)_+$, so
    $(y/G)\circ \beta$ is the colimit of $(y/G)\circ \alpha$.
    Therefore $y/G$ preserves filtered colimits.
\end{proof}

\begin{lemma}
    \label{lem:left-adjoint-preserves-compact-morphisms}
    Let $F : \catC \to \mathsf{D}$ be a functor having a right adjoint $G$
    that preserves filtered colimits.
    Then $F$ sends compact morphisms to compact morphisms.
\end{lemma}
\begin{proof}
    Let $f : x \to y$ be a compact morphism in $\catC$. The functor $x/F : x/ \catC \to Fx / \mathsf{D}$ has a right adjoint given by $\eta^* \circ Fx/G$, where $\eta$ is the unit of the adjunction $F \dashv G$.
    Since $\eta^*$ preserves filtered colimits by \Cref{lem:pres-filt-colims} and since $Fx/G$ preserves filtered colimits by
    \Cref{cor:if-g-preserves-filt-colim-then-so-do-slices-of-g}
    we have that $(x/F) f = Ff$ is a compact object in
    $Fx / \mathsf{D}$, i.e.\ $Ff$ is a compact morphism in $\mathsf{D}$.
\end{proof}

The following lemma is likely well-known to experts, but we provide a proof here for completeness.
\begin{lemma}
    \label{lem:compact-morphisms-closed-under-comp}
    Let $\catC$ be a category. The class of compact morphisms in $\catC$ is closed under composition.
\end{lemma}
\begin{proof}
    Let $f : x \to y$ and $g : y \to z$ be compact morphisms in $\catC$. Now let $\omega : I \to x/\catC$ be a filtered diagram in $x/ \catC$, and consider a morphism $\varphi : g \circ f \to \colim \omega$ in $x/\catC$.
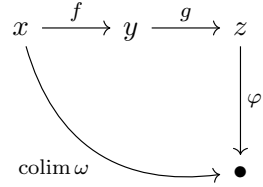
    \[\begin{tikzcd}[ampersand replacement=\&]
            x \& y \& z \\
            \\
            \&\& \bullet
            \arrow["f", from=1-1, to=1-2]
            \arrow["{\colim \omega}"', curve={height=30pt}, from=1-1, to=3-3]
            \arrow["g", from=1-2, to=1-3]
            \arrow["\varphi", from=1-3, to=3-3]
        \end{tikzcd}\]
    Since $f$ is compact, the morphism $\varphi \circ g : f \to \colim \omega$ factors through $\omega(i_0)$ for some $i_0 \in I$. Let $w_0$ be the codomain of $\omega(i_0)$, so that we have
    
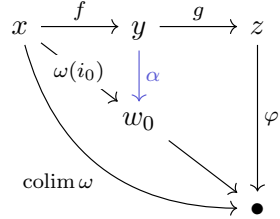
\[\begin{tikzcd}[ampersand replacement=\&]
            x \& y \& z \\
            \& {w_0} \\
            \&\& \bullet
            \arrow["f", from=1-1, to=1-2]
            \arrow["{\omega(i_0)}"{description}, from=1-1, to=2-2]
            \arrow["{\colim \omega}"', curve={height=30pt}, from=1-1, to=3-3]
            \arrow["g", from=1-2, to=1-3]
            \arrow["\alpha", color={rgb,255:red,92;green,92;blue,214}, from=1-2, to=2-2]
            \arrow["\varphi", from=1-3, to=3-3]
            \arrow[from=2-2, to=3-3]
        \end{tikzcd}\]
    for some morphism $\alpha$. Now we consider the composition
    \[\omega' := i_0/I \xrightarrow{i_0/\omega} \omega(i_0)/(x/\catC) \xrightarrow{\omega(i_0)/\cod} w_0/\mathsf{C} \xrightarrow{\alpha^*} y/\catC\]
    which is a filtered diagram in $y/C$. Note that the following diagram commutes:
    \[\begin{tikzcd}[ampersand replacement=\&]
            {i_0/I} \& {\omega(i_0)/(x/\catC)} \& {w_0/\catC} \& {y/\catC} \\
            I \&\&\& {x/\catC}
            \arrow["{i_0/\omega}"', from=1-1, to=1-2]
            \arrow["{\omega'}"{description}, curve={height=-30pt}, from=1-1, to=1-4]
            \arrow["\cod"', from=1-1, to=2-1]
            \arrow["{\omega(i_0)/\cod}", from=1-2, to=1-3]
            \arrow["\cod"', from=1-2, to=2-4]
            \arrow["{\alpha^*}", from=1-3, to=1-4]
            \arrow["{f^*}", from=1-4, to=2-4]
            \arrow["\omega"', from=2-1, to=2-4]
        \end{tikzcd}\]
    Thus, by \Cref{lem:pres-filt-colims} we have
    \[f^* \colim \omega' = \colim(f^* \circ \omega') = \colim(\omega \circ \cod),\]
    and since $I$ is filtered, the image of $\cod : i_0/I \to I$ is a final subcategory of $I$. Overall, this gives
    \[f^* \colim \omega' = \colim \omega,\]
    and in particular we note $\cod \colim \omega' = \cod (f^* \colim \omega') = \cod \colim \omega$.

    So, we have a diagram
    \[\begin{tikzcd}[ampersand replacement=\&]
            y \& z \\
            \& \bullet
            \arrow["g", from=1-1, to=1-2]
            \arrow["{\colim \omega'}"', from=1-1, to=2-2]
            \arrow["\varphi", from=1-2, to=2-2]
        \end{tikzcd}\]
    which we have not yet checked commutes. The commutativity is easy to check, however, because we know that $\colim \omega'$ factors through $\omega'(\id_{i_0})$, i.e. the lower triangle in the commutative square
    \[\begin{tikzcd}[ampersand replacement=\&]
            y \& z \\
            {w_0} \& \bullet
            \arrow["g", from=1-1, to=1-2]
            \arrow["{\omega(i_0)}"', from=1-1, to=2-1]
            \arrow["{\colim \omega'}"{description}, from=1-1, to=2-2]
            \arrow["\varphi", from=1-2, to=2-2]
            \arrow[from=2-1, to=2-2]
        \end{tikzcd}\]
    commutes. Now, since $g$ is a compact morphism in $\catC$, we know that $\varphi$ factors through $\omega'(\gamma)$ for some $\gamma \in i_0/I$. We have $\gamma : i_0 \to i_1$ for some $i_1 \in I$; letting $w_1$ denote the codomain of $\omega(i_1)$, we get
    \[\begin{tikzcd}[ampersand replacement=\&]
            x \& y \&\&\& y \&\& z \\
            \& {w_0} \&\&\&\& {w_1} \\
            \& {w_1} \&\&\&\&\& \bullet
            \arrow["f", from=1-1, to=1-2]
            \arrow["{\omega(i_0)}"{description}, from=1-1, to=2-2]
            \arrow["{\omega(i_1)}"{description}, curve={height=12pt}, from=1-1, to=3-2]
            \arrow["\alpha", from=1-2, to=2-2]
            \arrow["{\omega'(\gamma)}", curve={height=-30pt}, from=1-2, to=3-2]
            \arrow["g", from=1-5, to=1-7]
            \arrow["{\omega'(\gamma)}"{description}, from=1-5, to=2-6]
            \arrow["{\colim \omega'}"', curve={height=30pt}, from=1-5, to=3-7]
            \arrow["\beta"', color={rgb,255:red,92;green,92;blue,214}, from=1-7, to=2-6]
            \arrow["\varphi", from=1-7, to=3-7]
            \arrow["{\omega(\gamma)}"{description}, from=2-2, to=3-2]
            \arrow[from=2-6, to=3-7]
        \end{tikzcd}\]
    for some morphism $\beta$. Putting these diagrams together, we have
    \[\begin{tikzcd}[ampersand replacement=\&]
            x \&\& y \&\& z \\
            \&\&\& {w_1} \\
            \&\&\&\& \bullet
            \arrow["f", from=1-1, to=1-3]
            \arrow["{\omega(i_1)}"', from=1-1, to=2-4]
            \arrow["{\colim \omega' \circ f}"', curve={height=24pt}, from=1-1, to=3-5]
            \arrow["g", from=1-3, to=1-5]
            \arrow["{\omega'(\gamma)}"{description}, from=1-3, to=2-4]
            \arrow["\beta"', from=1-5, to=2-4]
            \arrow["\varphi", from=1-5, to=3-5]
            \arrow[from=2-4, to=3-5]
        \end{tikzcd}\]
    and we recall $\colim \omega' \circ f = f^* \colim \omega' = \colim \omega$. Thus, $\beta$ factors $\varphi$ through $\omega(i_1)$. Since $\varphi$ was arbitrary, we conclude that $g \circ f$ is a compact morphism in $\catC$, as desired.
\end{proof}

\begin{lemma}\label{lem:right-adjoints-transfer-AC-if-left-reflects-terminal-objects}
    Let $F : \catC \to \catD$ be left adjoint to $G : \catD \to \catC$. Suppose that:
    \begin{enumerate}
        \item $G$ preserves filtered colimits and initial objects,
        \item\label{assumption:F reflects terminal objs} $F$ reflects terminal objects (at least when restricted to compact objects),
        \item $\catD$ is Nullstellensatzian.
    \end{enumerate}
    Then $\catC$ is Nullstellensatzian.
\end{lemma}

\begin{proof}
    Since $\catD$ is Nullstellensatzian it has an initial object $0_\catD$.
    By the first assumption, $G0_\catD$ is initial in $\catC$, so
    $\catC$ has an initial object. Now let $x$ be a compact, nonterminal object in $\catC$. By the second assumption and \Cref{lem:left-adjoint-preserves-compacts}, $Fx$ is compact and nonterminal in $\catD$. Then by the third assumption, there exists a map $Fx\to 0_\catD$.
    Now the adjunction $F \dashv G$ gives a map
    $x\to G0_\catD$, and $G0_\catD$ is initial in $\catC$.
\end{proof}

The previous lemma is one tool for constructing Nullstellensatzian objects. We have another such tool:

\begin{proposition}
    \label{prop:right-adjoints-preserve-AC-if-left-reflects-terminal-objects}

    Let $\catC$ and $\catD$ be categories with terminal objects
    and let $F:\catC\to \catD$ be left adjoint to a functor
    $G:\catD\to \catC$ that preserves
    filtered colimits. Let $x$ be a Nullstellensatzian object in $\catD$. Suppose that:

    \begin{enumerate}
        \item $Gx$ is nonterminal.
        \item Precomposition with the
              counit $\epsilon_x^*:x/\catD\to FGx/\catD$
              has a left adjoint $\epsilon_{x*}$.
        \item The composite
              $\epsilon_{x*}\circ (Gx/F): Gx/\catC\to x/\catD$
              reflects terminal objects
              (at least when restricted to compact objects).
    \end{enumerate}
    Then $Gx$ is also Nullstellensatzian.
\end{proposition}

\begin{proof}
    $x$ is Nullstellensatzian by definition
    if and only if $x$ is nonterminal and
    $(x/\catD)$ is Nullstellensatzian,

    So since $Gx$ is assumed to be nonterminal,
    this follows from applying the previous theorem
    to the composite adjunction
    \[
        Gx/\catC \rightleftharpoons FGx/\catD
        \rightleftharpoons x/\catD.
    \]
    The left adjoint in the first adjunction is given
    by $Gx/F$ with right adjoint $\eta_{Gx}^*\circ (FGx/G)$.
    In the second adjunction the right adjoint
    is $\epsilon_x^*$ and the left adjoint is
    $\epsilon_{x*}$.

    The composite left adjoint reflects terminal objects by assumption,
    so we just need to check that the composite
    right adjoint preserves filtered colimits and
    sends the initial object to the
    initial object. By the triangle identity, the composite right adjoint
    is given by $x/G$ which preserves
    filtered colimits by \Cref{cor:if-g-preserves-filt-colim-then-so-do-slices-of-g}
    and sends $1_x$ to $1_{Gx}$
    as required.
\end{proof}

\begin{corollary}
    \label{cor:right-adjoints-preserve-AC-if-pushout-cond}
    Let $\catC$ and $\catD$ be categories with terminal objects
    and let $F:\catC\to \catD$ be left adjoint to a functor
    $G:\catD\to \catC$ that preserves
    filtered colimits. Let $x$ be a Nullstellensatzian object in $\catD$. Assume that
    \begin{enumerate}
        \item $G$ reflects terminal objects,
        \item $\catD$ has all pushouts.
    \end{enumerate}
    Then $Gx$ will be Nullstellensatzian if the following
    condition holds:

    For all compact morphisms $\alpha:Gx\to y$ if we form the pushout square
    \[
        \begin{tikzcd}
            FGx \arrow[d, "F\alpha"'] \arrow[r, "\epsilon_x"] & x \arrow[d] \\
            Fy \arrow[r]                                      & z
        \end{tikzcd}
    \]
    and $z$ is terminal in $\catD$ then $y$ is terminal in $\catC$.
\end{corollary}

\begin{proof}
    If $x$ is Nullstellensatzian, then it is nonterminal and since
    $G$ reflects terminal objects $Gx$ is nonterminal. So condition 1
    of the proposition holds. Also when $\catD$ has pushouts, then
    for any $f:a\to b$ in $\catD$, $f^*:b/\catD\to a/\catD$ has a left
    adjoint, which is given by pushout along $f$. Therefore
    $\epsilon_{x*}$ is given by pushout along $\epsilon_x$.
    So condition 2 holds. That then
    makes the pushout condition above a restatement of condition 3 of
    the proposition.
\end{proof}

\subsection{Examples of Important Functors that Preserve Nullstellsatzian Objects}

\begin{proposition}
    \label{prop:ring-of-funs-from-group-is-ac}
    For a finite group $G$, the
    functor $\Fun(G,-):\CRing\to \GCRing$ preserves Nullstellensatzian
    objects.
\end{proposition}
\begin{proof}
    $\Fun(G,-)$ is right adjoint to the forgetful functor, and
    if $\Fun(G,A)=0$, then $A=0$, so it reflects terminal objects.
    Moreover, $\GCRing$ and $\CRing$ both have all filtered colimits
    which are computed in $\Set$, and since $G$ is finite,
    $\Fun(G,-)$ commutes with filtered colimits in $\Set$ and therefore
    commutes with filtered colimits of commutative rings as well.

    Therefore we are in the situation of
    \Cref{cor:right-adjoints-preserve-AC-if-pushout-cond},
    so we just need to check the pushout condition.
    If $k$ is a commutative ring, then the counit $\epsilon_k$
    is evaluation at the identity element of $G$,
    $\ev_e:\Fun(G,k)\to k$.

    If $A$ is an algebra over $\Fun(G,k)$ in $\GCRing$, via a map
    $\alpha: \Fun(G,k)\to A$, then
    the idempotents $\delta_g \in \Fun(G,k)$ which are the
    indicator functions for each of the elements $g\in G$ induce a product
    decomposition
    \[ A=\prod_{g\in G} A\delta_g, \]
    and since $G$ acts transitively on the $\delta_g$, we have that 
    if $A_0 = A\delta_e$, then $A\cong \Fun(G,A_0)$.
    Moreover if we let $\alpha_0$ be the map
    $k\cong k\delta_e\to A\delta_e \cong A_0$,
    then $\alpha\cong \Fun(G,\alpha_0)$.

    Then the pushout square we get is
    \[
        \begin{tikzcd}
            {\Fun(G,k)} \arrow[d, "{\Fun(G,\alpha_0)}"'] \arrow[r, "\epsilon_k"] & k \arrow[d, "\alpha_0"] \\
            {\Fun(G,A_0)} \arrow[r, "\epsilon_{A_0}"]                            & A_0
        \end{tikzcd}
    \]
    and if $A_0$ is $0$, then so was $A$.
\end{proof}

\begin{lemma}
    \label{lem:fixed-pts-preserves-nullstellensatzian}
  The fixed point Tambara functor 
  $\underline{-} : \GCRing\to \GTamb$ preserves Nullstellensatzian
  objects.
\end{lemma}

\begin{proof}
  Let $S$ be a commutative ring with $G$-action.
  Since fixed points are finite limits they commute 
  with filtered colimits and arbitrary limits, 
  and we know that the fixed point Tambara functor 
  is right adjoint to $\ev_{G/e}$, and if 
  $\underline{S}=0$, then $S$ must have been zero 
  in the first place. So we just need to check the 
  pushout condition.

  If $A$ is an $\underline{S}$-algebra, then the pushout 
  is $A(G/e)\otimes_{\underline{S}(G/e)} S \cong A(G/e)$.
  So if the pushout is zero, then $A(G/e)$ is zero, 
  which implies that $A$ is zero.
\end{proof}

If we combine the previous two results with \Cref{lem:coinduction-from-e},
which says that the coinduction functor is the composite 
$C_e^G R = \underline{\Fun(G,R)}$, then we get the following corollary:

\begin{corollary}
    \label{cor:coind-preserves-ac}
    Let $G$ be a finite group. Then the coinduction 
    functor $C_e^G : \CRing \to \GTamb$ 
    preserves Nullstellensatzian objects.
\end{corollary}

Now, from our classification of Nullstellensatzian Tambara functors,
it will turn out that when $N$ is a normal subgroup of $G$, the evaluation functors 
$\ev_{G/N}:\GTamb\to\GCRing[G/N]$ preserve Nullstellensatzian objects. 
We will directly prove this in two cases, when $N=G$ (\Cref{prop:top-level-is-alg-closed}) and when $N=e$ (\Cref{prop:ev-at-underlying-preserves-nullstellensatzian}).
When $N=e$, the proof is able to exploit the fact that, 
in this case, evaluation and restriction are the same and restriction 
is also a left adjoint. 

While we could use the $N=e$ case to prove the
classification theorem, we will instead use the $N=G$ case 
because the proof in the $N=G$ case generalizes to incomplete Tambara 
functors with relatively small modifications. 

\begin{lemma}\label{lem:eval-has-left-adjt}
    For a finite $G$-set $X$, the functor $\ev_X:\GTamb\to \CRing$
    has a left adjoint. The same is also true for
    $\ev_{G/H}: \GTamb \to \GCRing[W_G(H)]$.
\end{lemma}

\begin{proof}
    Let $A$ be a commutative ring. We can (functorially) write $A$ as a coequalizer of
    free commutative rings,
    $\coeq \mathbb{Z}[x_j:j\in J]\rightrightarrows \mathbb{Z}[x_i:i\in I]$
    for some indexing sets $I$ and $J$ and maps of commutative rings.
    Then we can construct the exact same diagram in Tambara functors
    by adjoining the variables at level $X$ and define
    \[LA = \coeq \mathcal{A}_G[x_j^X:j\in J]\rightrightarrows \mathcal{A}_G[x_i^X:i\in I].\]
    Then for any Tambara functor $S$, we have that $\GTamb(LA,S)$
    is naturally isomorphic to maps $I\to S(X)$ picking out families of
    elements of $S(X)$ that satisfy the relations imposed by the
    coequalizer. And that is naturally isomorphic to $\CRing(A, S(X))$.
    Therefore a left adjoint exists.

    The same proof works when we remember the Weyl action.
\end{proof}

\begin{proposition}
    \label{prop:ev-at-underlying-preserves-nullstellensatzian}
  The functor $\ev_{G/e}: \GTamb\to \GCRing$ preserves 
  Nullstellensatzian objects.
\end{proposition}

\begin{proof}
  By the lemma above $\ev_{G/e}$ preserves filtered 
  colimits and has a left adjoint. It also reflects 
  terminal objects, because if $k(G/e)=0$, then 
  $1_G = \nm_e^G 1_e =0$, so $k=0$. So we just need 
  to verify the pushout condition.

  Let $L_e$ be the left adjoint to $\ev_{G/e}$ 
  constructed above. Then if $A$ is a $k(G/e)$ algebra,
  let $B:= L_e A \boxtimes_{L_e k(G/e)} k$ denote
  the pushout. We want to show that if $B$ is zero, 
  then $A$ must have already been zero. Now recall that 
  $\ev_{G/e}$ is also a left adjoint whose right adjoint 
  is the fixed point
  Tambara functor construction. 
  So $\ev_{G/e} B = (\ev_{G/e}L_e A)
  \boxtimes_{\ev_{G/e}L_e k(G/e)} k(G/e)$.
  However, if $S$ is any commutative ring with a $G$-action,
  $\ev_{G/e}L_e S \cong S$, since for any other 
  commutative ring with $G$-action, $T$, we have 
  the following natural isomorphisms 
  \[
    \GCRing(\ev_{G/e}L_e S, T)
    \cong 
    \GTamb(L_e S, \underline{T})
    \cong
    \GCRing(S, \ev_{G/e}\underline{T})
    =
    \GCRing(S, T),
  \]
  since $\ev_{G/e}\underline{T} = T^{e} = T$ 
  by definition.

  Therefore 
  $\ev_{G/e} B \cong A\boxtimes_{k(G/e)} k(G/e)\cong A$.
  So if $B$ is zero, $A$ is also zero.
  
\end{proof}

This proof is really more about the restriction functor than it is 
about the evaluation functor, and indeed 
a version of this proposition should also be true for 
slight modifications of the restriction 
functors $R^G_H: \GTamb\to \GTamb[H]$
where we modify the target category to keep track of the full $G$-action.

\subsection{Finitely Presented Fixed-Point Tambara Functors}

The existence of the coinduction functor also simplifies many calculations involving free algebras. We will record some of these here, leading to an important consequence: $S \mapsto \underline{S} : \CRing \to \GTamb$ preserves compact objects and compact morphisms.

\begin{proposition}
    \label{prop:underlying-of-free-on-underlying}
    Let $k$ be a Tambara functor. Then \[k[x_e](G/e) \cong k(G/e)[x_g | g \in G]\] as commutative rings.
\end{proposition}
\begin{proof}
    Let $k_e = k(G/e)$. We have
    \begin{align*}
        \kAlg[k_e]\of*{k[x_e](G/e), S} 
            & = \kAlg[k_e]\of*{R_e^G k[x_e], S}         \\
            & \cong \kAlg\of*{k[x_e], C_e^G S}             \\
            & \cong \of*{C_e^G S}(G/e)                          \\
            & = \Fun(G,S)                                   \\
            & \cong \kAlg[k_e]\of*{k_e[x_g | g \in G], S}
    \end{align*}
    natural in $k(G/e)$-algebras $S$.
\end{proof}



\begin{proposition}
    For any Tambara functor $k$, $k \boxtimes {-} : \GTamb \to \GTamb$ preserves levelwise surjections.
\end{proposition}
\begin{proof}
    This can be checked directly from the coend formula for the box product.
\end{proof}

\begin{corollary}\label{cor:constant-Z-is-fp}
    $\underline{\Z[x]}$ is a compact object in $\kAlg[\underline{\Z}]$.
\end{corollary}
\begin{proof}
    $\underline{\Z[x]}$ is generated as a $\underline{\Z}$-algebra (indeed, even as a Tambara functor) by the element $x$ at level $G/G$. Thus, we have a surjection $\underline{\Z}[x_G] \to \underline{\Z[x]}$.

    Next, we note that $\A \to \underline{\Z}$ is surjective, so \[\A[x_G] = \A \boxtimes \A[x_G] \to \underline{\Z} \boxtimes \A[x_G] = \underline{\Z}[x_G]\] is surjective. Since $\A[x_G]$ is levelwise finitely generated, we conclude that $\underline{\Z}[x_G]$ is also levelwise finitely generated (in particular, levelwise Noetherian). Thus, every ideal of $\underline{\Z}[x_G]$ is finitely generated. In particular, the kernel of $\underline{\Z}[x_G] \to \underline{\Z[x]}$ is finitely generated, so $\underline{\Z[x]}$ is compact in $\kAlg[\underline{\Z}]$.
\end{proof}

\begin{proposition}
    $S \mapsto \underline{S} : \CRing \to \kAlg[\underline{\Z}]$ preserves finite coproducts.
\end{proposition}
\begin{proof}
    For the empty coproduct we have directly that $\underline{\Z}$ is initial in $\kAlg[\underline{\Z}]$. \cite[Lemma 5.1]{LRZ-etale} shows that $\underline{S \otimes S'} \cong \underline{S} \boxtimes \underline{S'}$, and we must show that this agrees with $\underline{S} \boxtimes_{\underline{\Z}} \underline{S'}$. This is simply because $\underline{S} \boxtimes_{\underline{\Z}} \underline{S'}$ is defined as a coequalizer of two morphisms which are already equal---morphisms from $\underline{\Z}$ are unique whenever they exist, because such a morphism consists levelwise of ring homomorphisms from $\Z$.
\end{proof}

\begin{corollary}
    $\underline{\Z[x]}$ has the structure of a co-(commutative ring) object in the category of $\underline{\Z}$-algebras, coming from its co-(commutative ring) object structure in $\CRing$.
\end{corollary}

This tells us that $\kAlg[\underline{\Z}]\of*{\underline{\Z[x]},S}$ has a natural commutative ring structure for any $\underline{\Z}$-algebra $S$, i.e. $\kAlg[\underline{\Z}]\of*{\underline{\Z[x]},{-}}$ is a functor $\kAlg[\underline{\Z}] \to \CRing$.

\begin{proposition}
    $S \mapsto \underline{S} : \CRing \to \kAlg[\underline{\Z}]$ is left adjoint to $\kAlg[\underline{\Z}](\underline{\Z[x]}, {-})$.
\end{proposition}
\begin{proof}
    Given a morphism $\varphi : \underline{S} \to T$, define a morphism $\widetilde{\varphi} : S \to \kAlg[\underline{\Z}](\underline{\Z[x]},T)$ by
    \[\widetilde{\varphi}(s)_{G/H}(f) = f(\varphi_{G/H}(s)).\]
    It is easy to check that $\widetilde{\varphi}(s)$ is a homomorphism of Tambara functors for all $s$, and then that $\widetilde{\varphi}$ is a ring homomorphism. Thus, we obtain a function
    \[\varphi \mapsto \widetilde{\varphi} : \GTamb(\underline{S},T) \to \CRing(S, \kAlg[\underline{\Z}](\underline{\Z[x]}, T))\]
    which is clearly natural in $S$ and $T$. In the other direction, let $\psi : S \to \kAlg[\underline{\Z}](\underline{\Z[x]},T)$ be a ring homomorphism. Then define $\widehat\psi : \underline{S} \to T$ by
    \[\widehat\psi_{G/H}(s) = \psi(s)_{G/H}(x).\]
    We check that
    \[\widehat{\widetilde{\varphi}}_{G/H}(s) = \widetilde{\varphi}(s)_{G/H}(x) = \varphi_{G/H}(s)\]
    and
    \[\widetilde{\widehat{\psi}}(s)_{G/H}(x) = \widehat{\psi}_{G/H}(s) = \psi(s)_{G/H}(x),\]
    so these constructions are mutually inverse.
\end{proof}

\begin{corollary}
    $S \mapsto \underline{S} : \CRing \to \kAlg[\underline{\Z}]$ preserves compact morphisms.
\end{corollary}
\begin{proof}
    By \Cref{cor:constant-Z-is-fp}, $\kAlg[\underline{\Z}](\underline{\Z[x]},{-})$ preserves filtered colimits. By \Cref{lem:left-adjoint-preserves-compact-morphisms}, we are done.
\end{proof}

\begin{theorem}\label{thm:constant-preserves-compacts}
    $S \mapsto \underline{S} : \CRing \to \GTamb$ preserves compact objects and compact morphisms.
\end{theorem}
\begin{proof}
    Because $A/\kAlg[\underline{\Z}] \cong \kAlg[A] = A/\GTamb$, a morphism in $\kAlg[\underline{\Z}]$ is compact if and only if it is compact in $\GTamb$. Next, consider a compact object $R \in \CRing$. Then (equivalently) $\Z \to R$ is compact in $\CRing$, so $\underline{\Z} \to \underline{R}$ is compact in $\kAlg[\underline{\Z}]$, so $\underline{\Z} \to \underline{R}$ is compact in $\GTamb$. Since $\A \to \underline{\Z}$ is also compact (it is surjective), we conclude that $\A \to \underline{R}$ is compact, i.e., $\underline{R}$ is compact in $\GTamb$.
\end{proof}
\section{Classification of Nullstellensatzian Objects}

In this section we characterize the Nullstellensatzian $G$-Tambara functors and Nullstellensatzian $G$-rings, and study the algebraic $K$-theory of Nullstellensatzian Tambara functors. Additionally, we study algebraic closures of field-like $G$-Tambara functors.

\begin{definition}[\cite{BSY22}]
    A Tambara functor $k$ is Nullstellensatzian if every compact, nonterminal object in the slice category of Tambara functor algebras over $k$ admits a map to the initial $k$-algebra $\id_k : k \rightarrow k$.
\end{definition}

First, we collect some immediate consequences of this definition.


\begin{proposition}
    Let $k$ be a Nullstellensatzian Tambara functor. Then $k$ is field-like.
\end{proposition}
\begin{proof}
    Let $H \leq G$ and $x \in k(G/H)$ be arbitrary. If $(x)\ne k$, then 
    $q:k\to k/(x)$ is a nonzero, compact $k$-algebra, so it admits a retraction $r:k/(x)\to k$ such that $r\circ q = 1_k$. 
    Then $r(q(x)) = 0 = x$, so $x=0$. So $k$ has no proper nonzero ideals.
    
%
\end{proof}

\begin{corollary}
    Let $k$ be a Nullstellensatzian Tambara functor. Then $k(G/G)$ is a field.
\end{corollary}
\begin{proof}
    This is true of any field-like Tambara functor, see \cite[Remark 4.36]{Nak11a}.
\end{proof}

One might guess that if each level $k(G/H)$ is an algebraically closed field, then $k$ is Nullstellensatzian. However, this is not the case, as the following example shows. Roughly speaking, the correction is that one must consider $k(G/H)$ as a ring with Weyl group action, and ask for it to be Nullstellensatzian in the appropriate category of rings with group action.
\bigskip

\begin{example}
    \label{exl:constant-C-is-not-ac}
    If $G$ is not the trivial group, the constant Tambara functor at the algebraically closed field $\C$ is not a Nullstellensatzian Tambara functor.
\end{example}

\begin{proof}
    Consider the complex representation $\C[G]$ of $G$, which contains a copy of $\C$ in its fixed points via $1 \mapsto \sum_{g \in G} g$. Let $X := \C[G]/\C$, which is again a complex representation of $G$. Now $\Sym_{\C}(X)$ is a $G$-$\C$-algebra, and we can invert the elements of $G$ to obtain $Y := \Sym_{\C}(X)[g^{-1} : g \in G]$, which is again a $G$-$\C$-algebra.

    Since $G$ is not the trivial group, pick a function $f : G \to \C$ such that $\sum_{g \in G} f(g) = 0$ and $f(g) \neq 0$ for all $g \in G$. This yields a $\C$-linear function $X \to \C$, which then gives us a $\C$-algebra map $\Sym_{\C}(X) \to \Sym_{\C}(\C) \cong \C$. Since $f$ was nowhere-zero, this descends to a $\C$-algebra map $Y \to \C$, and thus we conclude that $Y \neq 0$.

    Now let $E := \underline{Y}$, which is a $\underline{\C}$-algebra. Since $Y$ is finitely presented, it is compact over $\C$, so (by \Cref{thm:constant-preserves-compacts}) $E$ is compact over $\underline{\C}$. Moreover, since $Y \neq 0$ we have $E \neq 0$. Now suppose towards a contradiction that there exists a morphism of Tambara functors $\varphi : E \to \underline{\C}$. In $E(G/e)$ we have $R_e^G T_e^G(e) = \sum_{g \in G} g = 0$, so we get
    \[\lvert G \rvert \varphi(e) = R_e^G T_e^G \varphi(e) = \varphi(R_e^G T_e^G e) = \varphi(0) = 0,\]
    whence $\varphi(e) = 0$. However, $e$ is a unit in $E(G/e) = Y$, so $\varphi(e)$ must be a unit, which is a contradiction.
\end{proof}

While \cref{exl:constant-C-is-not-ac} might be surprising at first, it offers the following heuristic. Nullstellensatzian objects must receive many maps. When $G$ acts trivially on something, any $G$-equivariant map must factor through $G$-orbits. However, the $G$-orbits of any coinduced Tambara functor are zero, and the only way to receive a morphism from the zero Tambara functor is to be zero. In light of this, the third-named author conjectures in \cite{Wis25} that the constant $\C$ Tambara functor actually is Nullstellensatzian if one works in the category of so-called \emph{clarified} $G$-Tambara functors. However, the results of \cite{Wis25} essentially say that this is the only possibility for field-like Nullstellensatzian clarified Tambara functors, and in fact there exist many Nullstellensatzian clarified Tambara functors which are not field-like.

\begin{lemma}
    \label{lem:nullstellensatzian-are-fixed-pt-tam-funcs}
    If $k$ is a Nullstellensatzian Tambara functor, then 
    the counit map $k\to \underline{k(G/e)}$ is an isomorphism.
\end{lemma}

\begin{proof}
    The components of the counit map $k\to \underline{k(G/e)}$,
    are the restriction maps 
    $\res^H_E:k(G/H)\to k(G/e)^H$. Since $k$ is a Tambara field
    these are injective, so we just need 
    to show that the restriction maps are also surjective
    to the fixed points. 

    Let $y\in k(G/e)^H$, and consider the $k$-algebra
    $k[x_{G/H}]/(\res^H_e x_{G/H} -y)$. This is nonzero, since it admits 
    a map to $\underline{k(G/e)}$ sending $x_{G/H}$ to $y$, and it is
    finitely presented, so it admits a retraction, $q$, to $k$.
    Then $\res^H_e q(x_{G/H}) = q(y) = y$ is a lift of $y$ to 
    $k(G/H)$.
\end{proof}


\begin{proposition}
\label{prop:top-level-is-alg-closed}
    The functor $\ev_*:\GTamb\to \CRing$ preserves Nullstellensatzian objects.
\end{proposition}


\begin{proof}
    Since filtered colimits are pointwise in Tambara functors,
    $\ev_{*}$ preserves filtered colimits. By \Cref{lem:eval-has-left-adjt}, $\ev_{*}$ is a right adjoint. If $\ev_*k=0$, then $k=0$, since for any subgroup $H$,
    $1_H\in k(G/H)$ is the restriction of $1_G\in k(G/G)=k(*)$, so
    if $1_G=0$, $1_H=0$ for all $H$, so $k=0$. Therefore $\ev_*$ reflects
    terminal objects.

    Thus we just need to check that when
    $\ev_{*}k\ne 0$ the pushout condition of \Cref{cor:right-adjoints-preserve-AC-if-pushout-cond} holds. Let $L_*$ be the left 
  adjoint to $\ev_*$. Then we need to show 
  that for a finitely presented $k(*)$ 
  algebra, $A$, if the pushout 
  $(L_*A)\boxtimes_{L_*k(*)}k$ is zero, 
  then $A$ must have been zero. Since $A$ is finitely 
  presented, we can write it as a polynomial algebra,
  $A
  \cong 
  k(*)[x_1,\ldots,x_n]/(f_1,\ldots,f_m)$,
  for some polynomials $f_1,\ldots, f_m$.
  Then we have that $L_*A \cong 
  L_*k(*)[x_{*,1},\ldots, x_{*,n}]/
  (f_1,\ldots,f_m)$, and the pushout is
  \[ (L_*A)\boxtimes_{L_*k(*)}k
  \cong 
  k[x_{*,1},\ldots,x_{*,n}]/(f_1,\ldots,f_m),\]
  where the $f_i$ are viewed as polynomials 
  in the $x_{*,1},\ldots,x_{*,n}$.

  Consider the Tambara functor
  $k[\underline{x}_1,\ldots,\underline{x}_n]$,
  which is isomorphic to the levelwise 
  tensor product of $k$ with $\Z[x_1,\ldots,x_n]$
  by \Cref{prop:free-green-has-norms-and-is-levelwise-tensor}. 
  Then we have a map 
  $k[x_{*,1},\ldots,x_{*,n}]
  \to k[\underline{x}_1,\ldots,\underline{x}_n]$
  given by sending $x_{*,i}$ to $\underline{x}_i$,
  so by taking quotients we can define a map from 
  the pushout, $k[x_{*,1},\ldots,x_{*,n}]/
  (f_1,\ldots,f_m)$, to 
  $k[\underline{x}_1,\ldots,\underline{x}_n]/
  (f_1,\ldots,f_m)$. Then we can complete the proof by showing that 
  $\ev_*(k[\underline{x}_1,\ldots,\underline{x}_n]/
  (f_1,\ldots,f_m)) \cong A$, since then if the pushout is zero,
  $k[\underline{x}_1,\ldots,\underline{x}_n]$ is zero, and therefore 
  $A$ is zero.

  To show this, let $I$ be the ideal of 
  $k(*)[x_1,\ldots,x_n]$ generated by 
  the polynomials $f_1,\ldots,f_m$
  and let $J$ be the ideal of 
  $k[\underline{x}_1,\ldots,\underline{x}_n]$
  generated by $f_1,\ldots,f_m$. 
  Certainly we have that $I\subseteq J(*)$,
  and we can complete the proof by showing 
  that this is actually an equality.

  Any element of $J(*)$ can be written
  as a sum of elements  
  of the form 
  \[
    \tr_H^G \of*{
      a \cdot
      \nm_K^H \res^G_K b
    },
  \]
  for subgroups $K\subseteq H\subseteq G$, 
  $a\in 
  \of*{
    k[\underline{x}_1,\ldots,\underline{x}_n]
  }(G/H)$
  and $b\in I$.
  so it suffices to show that these generators
  are all contained in $I$. 
  The key here is to show that 
  $\nm_K^H \res^G_K b = \res^G_H b^{[H:K]}$,
  since then 
  \[
    \tr_H^G \of*{
      a \cdot
      \nm_K^H \res^G_K b
    }
    = 
    \tr_H^G \of*{
      a \cdot
      \res^G_H b^{[H:K]}
    }
    = b^{[H:K]}\tr_H^G a \in I.
  \] 

  Then since $k$ is a Tambara field, 
  it satisfies the monomorphic restriction 
  condition, so 
  $k[\underline{x}_1,\ldots,\underline{x}_n]$
  also satisfies the monomorphic restriction 
  condition, since at level $G/H$ it is 
  isomorphic to $k(G/H)[x_1,\ldots,x_n]$.

  Therefore we can verify our identity 
  by applying $\res^H_e$ to both sides,
  which gives us
  \[
    \res^H_e \res^G_H b^{[H:K]}
    =
    \prod_{hK \in H/K} (hK)\res^K_e \res^G_K b
    = \res^H_e \nm_K^H\res^G_K b.
    \qedhere
  \]
\end{proof}



    



\begin{theorem}
    \label{theorem:alg-closed-iff-coinduced}
    A Tambara functor $k$ is Nullstellensatzian if and only if $k \cong C_e^G(\mathbb{F})$ for some algebraically closed field $\mathbb{F}$.
\end{theorem}

\begin{proof}

    The reverse direction is \cref{cor:coind-preserves-ac}. For the forwards direction, we begin by constructing a useful compact $k$-algebra. Let $K$ be the quotient of $k[x_{G/e}]$ by the following relations at level $G/e$ (recalling $k[x_{G/e}](G/e) \cong k(G/e)[x_g | g \in G]$ from \Cref{prop:underlying-of-free-on-underlying}):
    \begin{enumerate}
        \item (idempotence) $x_g^2 = x_g$
        \item (orthogonality) $x_g x_h = 0$ when $g \neq h$
        \item (completeness) $\displaystyle \sum_{g \in G} x_g = 1$
    \end{enumerate}
    To see that $K$ is nonzero, observe that we may construct a $k$-algebra map $K \rightarrow C_e^G k(G/e)$ as follows. First, choose $k[x_{G/e}] \rightarrow C_e^G k(G/e)$ classifying the choice of any idempotent which determines a projection $C_e^G k(G/e) \cong \Fun(G/e,k(G/e)) \rightarrow k(G/e)$. Next, observe that this map respects the three relations above, so that it descends to $K \rightarrow C_e^G k(G/e)$. Since the codomain of this morphism is nonzero (here we use the assumption that $k$ is nonzero), the domain is also nonzero.

    Since $K$ is a quotient of a finitely generated polynomial algebra by finitely many relations, $K$ is a compact $k$-algebra, and hence admits a section $K \rightarrow k$. Let $y_g \in k(G/e)$ denote the image of $x_g \in K(G/e)$. Then the collection of $y_g$ form a complete set of orthogonal idempotents in $k(G/e)$, hence determine an isomorphism $k(G/e) \cong \prod_{g \in G} k(G/e) \cdot y_g$. Since the 
    $y_g$ form a $G$-orbit, the rings $k(G/e)\cdot y_g$ are all 
    isomorphic, so if we let $A=k(G/e)\cdot y_e$, we can rewrite this 
    as an isomorphism 
    $k(G/e)\cong \Fun(G, A)$.

    Now by \Cref{lem:nullstellensatzian-are-fixed-pt-tam-funcs}
    $k\cong \underline{k(G/e)}$,
    so $k\cong \underline{\Fun(G,A)}$, and we know that 
    $C_e^GA \cong \underline{\Fun(G,A)}$, so this shows that 
    $k$ is coinduced. All that remains is to show that 
    $A$ is an algebraically closed field. 

    However, the diagonal map $\Delta: A\to \Fun(G,A)^G$ is an 
    isomorphism to the $G$-fixed points, so we have that 
    $k(*)\cong \Fun(G,A)^G \cong A$, and $k(*)$ is an algebraically 
    closed field by \cref{prop:top-level-is-alg-closed}.
\end{proof}



It is often considered a defect of the ideal-centric definition of field-like Tambara functors that there are examples of Tambara fields for which not every module is free. However, as pointed out to us by Mike Hill, the category of modules over the coinduction of a ring is naturally equivalent to the category of modules over the original ring (see \cref{prop:coinduction-and-modules} below). Thus this defect is miraculously absent for Nullstellensatzian Tambara functors, by the above theorem. In particular, if one attempts to define the K-theory of a Tambara functor in terms of its module category, then the K-theory of Nullstellensatzian Tambara functors is determined straightforwardly by the K-theory of algebraically closed fields. The following result is a discrete version of the results in \cite{BDS15}.

\begin{proposition}
    \label{prop:coinduction-and-modules}
    Let $R$ be a ring. Coinduction induces an adjoint equivalence between the category of $R$-modules and the category of $C_e^G R$-modules (defined as Mackey functor modules over the underlying Green functor of $C_e^G R$).
\end{proposition}

\begin{proof}
    Recall that any Tambara functor has the structure of a monoid with
    respect to the box product, and a module over a Tambara functor is a
    module over the underlying monoid. If $k$ is any $G$-Green functor and
    $M$ is a $k$-module, it follows that every abelian group $M(G/H)$ is a
    $k(G/H)$-module. If $k = C_e^G R$, then we may view $M(G/H)$ as an
    $R$-module through the restriction (i.e., the diagonal action).
    Note that the idempotents in $\Fun(G/H,R)$ determine an $R$-module
    decomposition \[ M(G/H) \cong \Fun(G/H,M_H) \] for some
    $R$-module $M_H$.

    The canonical map \[ M \rightarrow \underline{M(G/e)} \] of Mackey
    functors is surjective in every level by the double coset formula.
    Thus it suffices to show that all restriction maps for $M$ are
    injective. To see this, recall the Frobenius reciprocity relations
    \[ T_f(x) \cdot m = T_f(x \cdot R_f(m)) \] for every
    $f : G/e \rightarrow G/H$. In $C_e^G R$, the unit of each ring
    $C_e^G R(G/H)$ is in the image of the transfer.
    Choosing a preimage of the unit, $x$, we see that if
    $R_f(m) = 0$, then
    \[ m = T_f(x) \cdot m = T_f(x \cdot R_f(m)) = T_f(0) = 0 \]
    hence $R_f$ is injective, as desired.
\end{proof}

In fact, we may use the equivalence of \cref{prop:coinduction-and-modules} 
to give an alternative proof that the coinduction of any
algebraically closed field is Nullstellensatzian, bypassing the need
for \cref{lem:right-adjoints-transfer-AC-if-left-reflects-terminal-objects}.
Namely, observe that the Nullstellensatzian condition for an object
is stated in terms of a property of a certain category. 
To see that $C_e^G \mathbb{F}$ is Nullstellensatzian, for
$\mathbb{F}$ an algebraically closed field, it suffices to show
that the category of Tambara functor $C_e^G \mathbb{F}$ algebras
is equivalent to the category of $\mathbb{F}$-algebras.

Next, observe that a $C_e^G \mathbb{F}$-algebra is precisely a monoid in
the category of $C_e^G \mathbb{F}$-modules with respect to the box product
relative to $C_e^G \mathbb{F}$. Indeed, such a monoid is automatically a
Green functor, necessarily of the form $C_e^G R$, and since restriction
maps are injective, the double coset formula implies that there is
precisely one possible choice of norms. Similarly, an $\mathbb{F}$-algebra
is precisely a monoid in the category of $\mathbb{F}$-modules with respect
to the tensor product relative to $\mathbb{F}$. So the desired equivalence
of algebra categories follows from the following lemma.

\begin{lemma}
\label{lem:coind-equiv-is-sym-mon}
    The equivalence of \cref{prop:coinduction-and-modules} is strong symmetric monoidal (with respect to the tensor product over $\mathbb{F}$ and the box product relative to $C_e^G \mathbb{F}$).
\end{lemma}

\begin{proof}
    Let $C_e^G X$ and $C_e^G Y$ denote two arbitrary $C_e^G \mathbb{F}$-modules. Since their box product over $C_e^G \mathbb{F}$ is again a module, it is coinduced from a $\mathbb{F}$-module $Z$. Using the natural isomorphism \[ \Fun(G,Z) \cong \left( C_e^G X \boxtimes_{C_e^G \mathbb{F}} C_e^G Y \right)(G/e) \cong \Fun(G,X) \otimes_{\Fun(G,\mathbb{F})} \Fun(G,Y) \] we deduce a natural isomorphism $Z \cong X \otimes_{\mathbb{F}} Y$. Applying coinduction we obtain the desired natural isomorphism \[ C_e^G X \boxtimes_{C_e^G \mathbb{F}} C_e^G Y \cong C_e^G Z \cong C_e^G \left( X \otimes_{\mathbb{F}} Y \right) \textrm{.}\qedhere\]
\end{proof}

In fact, these results are strengthened greatly. In \cite{CW25} the third-named author and David Chan show that any coinduction $\mathrm{C}_H^G$ with $G$ abelian induces a symmetric monoidal equivalence on module categories. In \cite{Wis25} the third-named author removes the commutativity assumption, and extends the result to categories of Tambara functor algebras.

\begin{theorem}[\cite{CW25} \cite{Wis25}]
    Let $k$ be an $H$-Tambara functor. Then there is a symmetric monoidal equivalence 
    \[ \kMod \simeq \kMod[C_H^G k] \] 
    of categories with respect to the relative box product, and an 
    equivalence of categories 
    \[ \kAlg \simeq  \kAlg[C_H^G k] \] of algebras.
\end{theorem}

\begin{corollary}
    If $\mathbb{F}$ is algebraically closed, then $C_e^G \mathbb{F}$ is Nullstellensatzian.
\end{corollary}

Since we visibly never used the norms in this proof, it clearly applies to incomplete Tambara functors as well.

\begin{corollary}
\label{cor:K-theory-equiv}
    If $k \cong C_e^G \mathbb{F}$ is Nullstellensatzian, then the equivalence of \cref{prop:coinduction-and-modules} induces an equivalence of $\mathbb{E}_\infty$-ring spectra $K(k) \cong k(\mathbb{F})$.
\end{corollary}

We may directly deduce a classification of Nullstellensatzian $G$-rings from \cref{theorem:alg-closed-iff-coinduced}, although a similar argument (building a compact algebra whose section forces structure) also goes through. We give this alternate approach to demonstrate the robustness of \cref{lem:right-adjoints-transfer-AC-if-left-reflects-terminal-objects}.

\begin{theorem}
    A $G$-ring $k$ is Nullstellensatzian if and only if it is the coinduction of an algebraically closed field $\mathbb{F}$, $k \cong \Fun(G,\mathbb{F})$.
\end{theorem}

\begin{proof}
    By \Cref{prop:ring-of-funs-from-group-is-ac}, every $G$-ring of the 
    form $\Fun(G,\mathbb{F})$ for $\mathbb{F}$ an algebraically closed 
    field is Nullstellensatzian. 
    For the converse, if a $G$-ring $k$ is 
    Nullstellensatzian, then
    \Cref{lem:fixed-pts-preserves-nullstellensatzian} tells us that 
    $\underline{k}$ is Nullstellensatzian. Therefore 
    $\underline{k}\cong C_e^G(\mathbb{F})$ for some algebraically 
    closed field $\mathbb{F}$ by \Cref{theorem:alg-closed-iff-coinduced}, and evaluating at $G/e$ on both sides 
    gives that $k\cong \Fun(G,\mathbb{F})$.
\end{proof}

Presumably this result yields a $K$-theoretic statement analogous to \cref{cor:K-theory-equiv} in the context of Merling's genuine $K$-theory spectrum of a ring with $G$-action \cite{Mer17}, although the authors have not determined the precise form of this relationship. It is likely either a statement about $G$-fixed points or the underlying nonequivariant spectrum of Merling's construction.

Finally, we aim to characterize ``the'' algebraic closure of a field-like Tambara functor. As in the nonequivariant case, it turns out that algebraic closures are unique up to isomorphism. Recall \cite[Remark 4.36]{Nak11a} that if $k$ is a field-like Tambara functor, then the fixed-points $k(G/e)^G$ are a field.

\begin{theorem}
    \label{theorem:characterization-of-algebraic-closure}
    Assume $k$ is field-like. There is a map $k \rightarrow C_e^G \left( \overline{k(G/e)^G} \right)$, where the codomain is the coinduction of an algebraic closure of the field $k(G/e)^G$, and any map from $k$ to a Nullstellensatzian Tambara functor $K$ factors (noncanonically) through this map.
\end{theorem}

\begin{proof}
    Fix a choice of algebraic closure $k(G/e)^G \rightarrow \overline{k(G/e)^G}$. Rotating the coinduction-restriction adjunction, we obtain $C_e^G \left( \overline{k(G/e)^G} \right) $.

    Now let $k \rightarrow K$ any map with $K$ algebraically closed. Then $K \cong C_e^G \mathbb{F}$ for some algebraically closed field $\mathbb{F}$. Rotating the coinduction-restriction adjunction and applying fixed-points, we obtain a map \[ k(G/e)^G \subset k(G/e) \rightarrow \mathbb{F} \textrm{.} \] Since $k(G/e)$ is integral over $k(G/e)^G$ (by \cite[Exercise 5.12]{AM69}), this map factors through \[ k(G/e) \rightarrow \overline{k(G/e)^G} \rightarrow \mathbb{F} \textrm{,} \] where the first map is our fixed choice of algebraic closure of $k(G/e)^G$. Applying $C_e^G$ to the second map yields a morphism of Nullstellensatzian Tambara functors which manifestly factors $k \rightarrow K$ on the $G/e$ level. Since all objects are fields, all restriction maps are injective, so checking that $k \rightarrow C_e^G \overline{k(G/e)^G}$ factors $k \rightarrow K$ may be done on the $G/e$ level.
\end{proof}

For this reason, we refer to (a fixed choice of map) $k \rightarrow C_e^G \overline{k(G/e)^G}$ as algebraic closure of $k$. Of course, if $k$ is not a field, then there may be nonisomorphic choices of algebraic closures; this may be seen in the nonequivariant setting, for example, by consideration of the maps $\mathbb{Z} \rightarrow \mathbb{F}_p \rightarrow \overline{\mathbb{F}_p}$ for $p$ any prime.

Similarly, one may check that for any $G$-ring of the form $Fun(G/H,\mathbb{F})$, there is a unique algebraic closure, up to noncanonical isomorphism, given by $\Fun(G,\overline{\mathbb{F}})$, for $\overline{\mathbb{F}}$ an algebraic closure of $\mathbb{F}$.

\subsection{Incomplete Tambara Functors}
\label{subsec:incomplete-Tamb-functors}

We can generalize Tambara functors by replacing $\Poly_G$ with
other categories of a similar flavor (i.e. whose morphisms are bispans). In full generality there are many interesting categories that one can use in place of $\Poly_G$ \cite{benthesis}, but the interesting generalization for our purposes will
be to consider categories $\Poly_{G,\mathcal{O}_\times}$,
where $\mathcal{O}_\times$ is a subcategory of $G\mhyphen\set$ which is wide,
pullback-stable, and a symmetric monoidal subcategory with
respect to the coproduct. $\Poly_{G,\mathcal{O}_\times}$ is the subcategory
of $\Poly_G$ consisting of diagrams
$X\xleftarrow{h} A \xrightarrow{g} B \xrightarrow{f} Y$ such that $g\in\mathcal{O}_\times$.
That this is a subcategory can be checked using pullback stability and
the formulas for composing the generating morphisms above. We require
$\mathcal{O}_\times$ to be a symmetric monoidal subcategory with
respect to the product so that $\Poly_{G,\mathcal{O}_\times}$
has all products and they agree with the products in
$\Poly_G$, which are given by disjoint union of $G$-sets.

Then we can define $\Poly_{G,\mathcal{O}_\times}$-Tambara functors
in exactly the same way as above. These are called \emph{incomplete Tambara
    functors} because they also admit a description as in \Cref{prop:alt-def-of-Tambara-functors}, with exactly one difference: we only have norm maps $\nm_H^K$  when the map
$G/H\to G/K$ lies in $\mathcal{O}_\times$. Incomplete Tambara functors were originally introduced in \cite{BH18}, and we refer the reader to this source for a more complete exposition.

One might wonder the extent to which our results apply to incomplete Tambara functors, and more generally bi-incomplete Tambara functors (those are Tambara functors which in addition to not having all norms, do not necessarily have all transfers). In fact, as we have remarked, the coinduction of an algebraically closed field is always a Nullstellensatzian incomplete Tambara functor. However, there are Nullstellensatzian objects not of this form.

To see this, consider the category of ``maximally" bi-incomplete $G$-Tambara functors, ie those which have no norms and no transfers. The category of such is just the category of coefficient systems of rings. In this category, form the unique bi-incomplete Tambara functor $k$ whose $G/G$-level is an algebraically closed field $\mathbb{F}$, and whose remaining levels are zero. The slice category of bi-incomplete Tambara functor algebras over $k$ may be identified with the category of $\mathbb{F}$-algebras. In particular, this implies $k$ is Nullstellensatzian.

One may be able to use \cref{lem:coinduction-from-e} to give a classification of Nullstellensatzian bi-incomplete Tambara functors. The authors plan to investigate this in future work.

\printbibliography
\end{document}